\documentclass[a4paper,reqno,oneside]{amsart}

\usepackage{bm,amsfonts,amsmath,amssymb,mathrsfs,bbm,amsthm} 
\usepackage{graphicx, color,hyperref,eurosym, eucal,palatino}
\usepackage{graphicx, color,hyperref,eurosym, eucal}
\usepackage[normalem]{ulem}

\usepackage{hyperref}
\usepackage{enumitem}
\usepackage{xcolor, amssymb}
\usepackage{marginnote}
\usepackage[normalem]{ulem}
\usepackage{eurosym}
\usepackage{stackrel}
\definecolor{darkred}{rgb}{0.65,0,0}
\definecolor{darkblue}{rgb}{0,0,0.65}

\numberwithin{figure}{section}
\numberwithin{equation}{section}
\theoremstyle{plain}
\newtheorem*{thm*}{Theorem}
\newtheorem{thm}{Theorem}[section]

\newtheorem{lem}[thm]{Lemma}
\newtheorem{prop}[thm]{Proposition}

\newtheorem{dfn}[thm]{Definition}
\theoremstyle{remark}

\theoremstyle{plain}

\newtheorem*{res0}{Main result for convex domains}
\newtheorem*{res1}{Main result for nearly circular domains}
\newtheorem*{res2}{Reversed Faber-Krahn}
\newtheorem{corol}[thm]{Corollary}
\theoremstyle{definition}

\newtheorem{rem}[thm]{Remark}

\newcommand\cyl{{\rm cyl}}
\newcommand\rmi{{\rm i}}
\newcommand\rmo{{\rm o}}
\newcommand\rmc{{\rm c}}
\newcommand\rms{{\rm s}}

\newcommand\cR{\mathcal{R}}

\newcounter{counter_a}
\newenvironment{myenum}{\begin{list}{{\rm(\roman{counter_a})}}%
		{\usecounter{counter_a}
			\setlength{\itemsep}{1.ex}\setlength{\topsep}{0.8ex}
			\setlength{\leftmargin}{5ex}\setlength{\labelwidth}{5ex}}}{\end{list}}

\newcommand\wt{\widetilde}
\newcommand\eps{\varepsilon}
\newcommand\lm{\lambda}
\newcommand\p{\partial}
\def\sfK{{\mathsf K}}
\newcommand\sfT{{\mathsf T}}
\newcommand\Omg{\Omega}

\newcommand{\bx}{x}

\newcommand{\bo}{0}
\newcommand{\bn}{\mathbf{n}}
\newcommand{\be}{\mathbf{e}}
\newcommand{\bd}{\mathbf{d}}

\newcommand\ii{{\mathsf{i}}}
\newcommand\dl{\delta}
\newcommand\s{\sigma}
\newcommand\tm{\times}

\newcommand{\C}{\mathbb{C}}
\newcommand{\R}{\mathbb{R}}

\newcommand{\T}{\mathbb{T}}
\newcommand{\D}{\mathbb{D}}
\newcommand{\I}{\mathbb{I}}
\newcommand{\Z}{\mathbb{Z}}
\newcommand{\N}{\mathbb{N}}
\newcommand{\sfD}{\mathsf{D}}
\newcommand\arr{\rightarrow}
\newcommand\kp{\kappa}
\newcommand\dd{{\mathsf{d}}}
\newcommand{\spn}{\mathsf{span}\,}

\def\cF{{\mathcal F}}
\def\cG{{\mathcal G}}

\def\cE{{\mathcal E}}

\def\cH{{\mathcal H}}
\def\cN{{\mathcal N}}
\def\cD{{\mathcal D}}
\def\cI{{\mathcal I}}
\def\cL{{\mathcal L}}
\def\frb{{\mathfrak b}}
\def\frt{{\mathfrak t}}
\newcommand\frq{\mathfrak{q}}
\newcommand\dom[1]{\mathrm{dom}\left(#1\right)}
\newcommand\spe[1]{\mathsf{Sp}\left(#1\right)}
\newcommand\sped[1]{\mathsf{Sp}_{\rm d}\left(#1\right)}
\newcommand\speess[1]{\mathsf{Sp}_{\rm ess}\left(#1\right)}

\renewcommand\gg{\gamma}
\renewcommand\tt{\theta}

\newcommand\Op{\sfD_\Omg}
\newcommand\OpD{\sfD_\D}
\newcommand\ov{\overline}

\title[Spectral gap for graphene quantum dots]{
A sharp upper bound on the spectral gap for graphene quantum dots}

\author{Vladimir Lotoreichik}
\address{Department of Theoretical Physics, Nuclear Physics Institute, 	Czech Academy of Sciences, 25068 \v Re\v z, Czech Republic}
\email{lotoreichik@ujf.cas.cz}
\urladdr{http:/gemma.ujf.cas.cz/~lotoreichik}

\author{Thomas Ourmi\`eres-Bonafos}
\address{CNRS \& Universite Paris-Dauphine, PSL Research University, CEREMADE,
Place de Lattre de Tassigny, 75016 Paris, France}

\email{ourmieres-bonafos@ceremade.dauphine.fr}
\urladdr{http://www.ceremade.dauphine.fr/~ourmieres/}

\begin{document}

\begin{abstract}
The main result of this paper is a sharp upper bound on the first positive eigenvalue of Dirac operators in two dimensional simply connected $C^3$-domains with infinite mass boundary conditions. This bound is given in terms of a conformal variation, explicit geometric quantities and of the first eigenvalue for the disk. 
Its proof relies on the min-max
principle applied to the squares of these Dirac operators. A suitable test function is constructed by means of a conformal map. This general upper bound involves the norm of the derivative of the underlying conformal map in the Hardy space $\cH^2(\D)$. Then, we apply known estimates of this norm for convex and for nearly circular,
star-shaped domains in order to get explicit geometric
upper bounds on the eigenvalue.
These bounds can be re-interpreted as reverse Faber-Krahn-type inequalities under adequate geometric constraints.
\end{abstract}                                                            

\maketitle

\section{Introduction}

\subsection{Motivations and statement of the main result}
The Dirac operator defined on a bounded domain of the Euclidean space $\R^2$ attracted a lot of attention in the recent few years. Motivated by the unique properties of low energy charge carriers in grap\-hene, various mathematical questions related to these Dirac operators have arisen, and some of them have been dealt with very recently.

The question of self-adjointness is addressed, for instance, for a large class of local boundary conditions in~\cite{BFSV17a} and it covers the particular boundary conditions commonly used in the physics literature~\cite{AB08}: the so-called \emph{zigzag}, \emph{armchair}, and \emph{infinite mass} boundary conditions.

The next step is to investigate the spectral properties of these models. For instance, the spectrum of the massless Dirac operator in a bounded domain with \emph{zigzag} boundary conditions is studied in~\cite{S95}. It turns out that this spectrum exhibits an interesting behaviour: it consists of the eigenvalue $0$, being of infinite multiplicity, and of a sequence of discrete eigenvalues related to the one of the Dirichlet Laplacian in the same domain.

The structure of the spectrum of the massless Dirac operator on a bounded domain with \emph{infinite mass} boundary conditions has a different flavour. Indeed, the model is now
invariant under charge conjugation, which implies the symmetry of the spectrum with respect to the origin (moreover, this spectrum is discrete).

Note that \emph{infinite mass} boundary conditions
for the Dirac operator
arise when one considers the Dirac operator on the whole Euclidean plane $\R^2$ with an ``infinite mass"
 outside a bounded domain and zero mass
inside it. This is mathematically justified in~\cite{BCTS18,SV18} (see also~\cite{ALMR18} for a three-dimensional version and \cite{MOBP18} for a generalization to any dimension). For this reason, these boundary conditions can be viewed as the relativistic counterpart of Dirichlet boundary conditions for the Laplacian.

It is well known that for partial differential operators defined on domains the shape of the domain manifests in the spectrum.
In particular, bounds on the eigenvalues can be given in terms of various geometrical quantities. In many
cases, it is also known that the ball (the disk, in two dimensions) optimizes the lowest eigenvalue under reasonable geometric constraints.
For example, the famous Faber-Krahn inequality
for Dirichlet Laplacians
(formulated
in two dimensions) states that
\begin{equation}\label{eq:FK}
	\lm_1(\Omega) \geq \lm_1(\D)
\end{equation}
for all Lipschitz domains $\Omega\subset\R^2$ of the same area as the unit disk $\D$ 
(see~\cite{F23} and~\cite{K25}); here $\lm_1(\Omega)$ denotes the first eigenvalue of the Dirichlet Laplacian on $\Omega$. 
In the same spirit, for any convex domain $\Omega\subset\R^2$, it is proven 
in~\cite[\S 5.6]{PS51} and in~\cite[Theorem 2]{FK08} that a reverse Faber-Krahn-type inequality with a geometric pre-factor
\begin{equation}\label{eq:PS}
	\lm_1(\Omg) 
	\leq 
	\frac{|\p\Omg|}{2\rho_{\rm i}|\Omg|}\lm_1(\D),
\end{equation}
holds where $\rho_{\rm i} > 0$ is the inradius of $\Omg$, $|\Omg|$ denotes the area of $\Omg$ and $|\p\Omg|$ stands for its perimeter.  Related upper bounds for the lowest Dirichlet eigenvalue are obtained e.g. in~\cite{PW61, P60}, see also the numerical study~\cite{AF06}.
Further spectral optimization results for the Dirichlet Laplacian can be found in the monographs~\cite{Henrot,Henrot2}; see also the references therein.

For the two-dimensional massless Dirac operator
$\Op$ with infinite mass boundary conditions
on a bounded, simply connected, $C^2$-domain $\Omega$
a lower bound on the principal eigenvalue is given in~\cite{BFSV17b} and reads in the case of \emph{infinite mass} boundary conditions as
\begin{equation}\label{eqn:lowboundbeng}
	\mu_1(\Omega) > \sqrt{\frac{2\pi}{|\Omega|}},
\end{equation}
where $\mu_1(\Omega)$ is the first non-negative eigenvalue of $\Op$. This bound is easy to compute and it
yields an estimate on the size
of the spectral gap. However, it
is not intrinsically Euclidean,
because the equality in~\eqref{eqn:lowboundbeng}
is not attained on any $\Omega\subset\R^2$.	
It is not yet known whether for $\Op$ a direct analogue of
the lower bound as in the Faber-Krahn inequality~\eqref{eq:FK} holds.

One should also mention numerous results in the differential geometry literature, where lower  and upper bounds have been found for Dirac operators on two-dimensional manifolds without boundary (see for instance~\cite{B92} and \cite{AF99, B98}). In~\cite{R06}, manifolds with boundaries are investigated and note that the mentioned \emph{CHI} (chiral) boundary conditions correspond to our \emph{infinite mass} boundary conditions. For two-dimensional manifolds, the author of~\cite{R06} provides a lower bound on the first eigenvalue which is actually~\eqref{eqn:lowboundbeng}. We
remark that
upon passing to the more general setting of manifolds the equality in~\eqref{eqn:lowboundbeng}
is attained on hemispheres.

Using the min-max principle and the estimate~\eqref{eq:PS}
one can easily show the following upper
bound
\begin{equation}\label{eq:easy_bound}
	\mu_1(\Omg) \le \sqrt{\lm_1(\Omg)}
	\le \left(	\frac{|\p\Omg|}{2\rho_{\rm i}|\Omg|}\lm_1(\D)\right)^{1/2};
\end{equation}
cf. Proposition~\ref{prop:var_form}.
This bound has a concise form, but it is not tight in particular cases.
Especially, for domains that are close to a disk  the bound~\eqref{eq:easy_bound} is not sharp, since $\mu_1(\D) \approx 1.4347$
and $\sqrt{\lm_1(\D)} \approx 1.5508$.

To our knowledge, there is no upper bound on 
$\mu_1(\Omg)$ expressed in terms of explicit geometric quantities, which is tight for domains being close to a disk. This is the question we tackle in this paper for the case of $C^3$-domains.
The inequalities that we obtain can be viewed as natural counterparts of~\eqref{eq:PS} in this new setting and our results roughly read as follows (see Theorems~\ref{thm:mainthrig} and~\ref{thm:mainthrig2} for rigorous statements).
\begin{res0}\label{thm:mainthvague}
	Let $\Omega\subset\R^2$ be a bounded, convex, $C^3$-domain with $0\in\Omega$ and let $\mu_1(\Omega)$ be the first non-negative eigenvalue of the massless Dirac operator $\Op$ with \emph{infinite mass} boundary conditions. Then, there is an explicitly given geometric functional
	$\cF_\rmc(\cdot)$ such that
	\begin{equation}\label{eq:main_ineq}
		\cF_\rmc(\Omega)\mu_1(\Omega) \leq \cF_\rmc(\D_r) \mu_1(\D_r),
	\end{equation}
	where $\D_{r}$ is the disk
	of radius $r > 0$ centered at the origin and $\cF_\rmc(\D_r) = r$ holds. Moreover, the inequality~\eqref{eq:main_ineq} is strict
	unless $\Omg = \D_{r'}$ for some $r' > 0$.
\end{res0}
\begin{dfn}\label{dfn:nc}
A bounded,
$C^3$-domain $\Omega\subset\R^2$, which is star-shaped with respect to the origin and which is parametrized in polar coordinates by $\rho = \rho(\phi)$, is called \emph{nearly circular} if
\begin{equation}\label{eq:rho_star}
	\rho_\star = \rho_\star(\Omg) :=\sup\bigg(\frac{|\rho'|}{\rho}\bigg)
	< 1.
\end{equation}
\end{dfn}
\begin{res1}\label{thm:mainthvague}
	Let $\Omega\subset\R^2$ be a bounded
	$C^3$-domain, which is nearly circular
	in the sense of Definition~\ref{dfn:nc}.
	Let $\mu_1(\Omega)$ be the first non-negative eigenvalue of the massless Dirac operator $\Op$ with \emph{infinite mass} boundary conditions. Then, there is an explicitly given geometric functional
	$\cF_\rms(\cdot)$ such that
	\begin{equation}\label{eq:main_ineq1}
	\cF_\rms(\Omega)\mu_1(\Omega) \leq \cF_\rms(\D_r)\mu_1(\D_r),
	\end{equation}
	where $\D_{r}$ is the disk
	of radius $r > 0$ centered at the origin and $\cF_\rms(\D_r) = r$ holds. Moreover, the inequality~\eqref{eq:main_ineq} is strict
	unless $\Omg = \D_{r'}$ for some $r' > 0$.
\end{res1}
The Dirac operator $\Op$ and the functionals
$\cF_\rmc$, $\cF_\rms$ appearing in~\eqref{eq:main_ineq},~\eqref{eq:main_ineq1} are rigorously defined further on, namely,
in Definition~\ref{def:Op} and
Equations~\eqref{eqn:functionalmainth},~\eqref{eqn:functionalmainth2}, respectively.
The main results are then precisely formulated
in Theorems~\ref{thm:mainthrig},~\ref{thm:mainthrig2}. 
Before going any further, let us comment on the assumptions and inequalities~\eqref{eq:main_ineq},
\eqref{eq:main_ineq1}.
\begin{rem} Even though for convex polygonal domains the Dirac operator $\Op$ can be defined in a similar fashion as for $C^3$-domains (see~\cite{LO18}), it will be clear from the proof that certain smoothness assumption on the domain $\Omega$ seems to be crucial for our results to hold. 
However, we expect that the smoothness hypothesis on $\Omega$ can be relaxed from $C^3$ to $C^2$-smoothness with additional efforts.
\end{rem}

\begin{rem} The strategy relying on a so-called \emph{invertible double} discussed in \cite[\S 2]{BFSV17b} (see also \cite[Chapter 9]{BW93}) might also yield new upper bounds using the known ones for two-dimensional manifolds without boundary. We do not discuss it here, first in order to keep a self-contained and elementary proof and, second, to obtain a result in terms of explicit geometric quantities:
the area $|\Omega|$, the maximal (non-signed) curvature $\kappa_\star$ of $\p\Omega$ and of the radii
$r_{\rmi} = \min_{\bx\in\p\Omg}|\bx|$, $r_{\rmo} = \max_{\bx\in\p\Omg}|\bx|$. Namely,
$\cF_\rmc$ is a function of all these parameters
and the parameter $\rho_\star$ introduced in~\eqref{eq:rho_star} plays a role in the definition of $\cF_\rms$.\end{rem}

Our main results imply two reverse Faber-Krahn-type
inequalities for the Dirac operator $\Op$. Indeed, let us denote by $\mathcal{E}_\rmc$ the set of bounded, convex $C^3$-domains $\Omega$ containing the origin and by $\mathcal{E}_\rms$ the set of bounded, nearly circular $C^3$-domains. Then, the following holds.
\begin{res2}
	Let $\sharp \in \{\rmc,\rms\}$ and $\Omega \in \mathcal{E}_{\sharp}$ such that $\cF_{\sharp}(\Omega) = r > 0$
	with $\Omg \ne\D_r$. Then the following
	inequality holds
	\begin{equation*}\label{eq:main_ineq2}
		\mu_1(\Omega) < \mu_1(\D_r).
	\end{equation*}
\end{res2}	
All the geometric bounds we obtain are consequences
of the following estimate which holds for any
bounded, simply connected, $C^3$-domain $\Omg\subset\R^2$ with $0\in\Omg$:
\begin{equation}\label{eq:abstract_bound_intro}
	\mu_1(\Omega) 
	\le 
	\left(
		\frac{2\pi}{|\Omega| + \pi r_\rmi^2}
	\right)^{1/2}
	\kp_\star \|f'\|_{\cH^2(\D)}\mu_1(\D),
\end{equation}
where $f\colon \D\arr\Omg$ is a conformal map
with $f(0) = 0$
and $\|f'\|_{\cH^2(\D)}$ is the norm of its derivative
in the Hardy space $\cH^2(\D)$.
The equality in~\eqref{eq:abstract_bound_intro}
occurs if, and only if, $\Omg = \D_{r'}$ for some $r' > 0$.
This abstract
bound is obtained in Theorem~\ref{thm:mainthrig0}.

%

\subsection{Strategy of the proof}

The proof is decomposed into four steps. First, thanks to the symmetry of the spectrum for the Dirac operator $\Op$ we compute the quadratic form of its square and characterize the squares of its eigenvalues \emph{via} the min-max principle.

Second, following the strategy of~\cite{S54}, we use a conformal map from the unit disk $\D$ onto the domain $\Omega$ in order to reformulate the min-max principle characterizing the first non-negative eigenvalue.

Third, we evaluate the corresponding Rayleigh quotient for a special test function that we construct by means of the first mode of the Dirac operator $\sfD_\D$ on the unit disk $\D$.

Finally, it remains to estimate each term
in this Rayleigh quotient in terms of suitable geometrical quantities. However, as the structure of the Dirac operator $\Op$ is more sophisticated than the one of the Neumann Laplacian investigated in~\cite{S54}, we have to control several additional terms. One of them involves  
the norm in the Hardy space $\cH^2(\D)$ of the derivative of the employed conformal map. We handle this term using 
available geometric estimates
for convex domains~\cite{K17} and	
for nearly circular domains~\cite{G62}.
In fact, other ways to control geometrically this 
Hardy norm are expected to yield new inequalities.

\subsection{Structure of the paper}
In Section~\ref{sec:diropdef} we rigorously define the Dirac operator $\Op$ and recall known results about it.
Section~\ref{sec:varchar} is devoted to the derivation of a variational characterization for the eigenvalues of  $\Op$. 
After precisely stating the main result in Theorem~\ref{thm:mainthrig}, we prove it in Section~\ref{sec:mainth}.

The paper is complemented by two appendices,  which are provided for completeness and convenience of the reader.
Appendix~\ref{app:disk} is about the eigenstructure of the disk and Appendix~\ref{app:prop}
deals with a geometric result
regarding the functional $\cF_\rmc$ on domains
with symmetries.

\section{The massless Dirac operator with infinite mass boundary conditions}\label{sec:diropdef}

This section is decomposed as follows. In \S \ref{ref:not} we introduce a few notation that will be used all along this paper and \S \ref{subsec:defDirope} contains the rigorous definition of the massless Dirac operator with \emph{infinite mass} boundary conditions as well as its basic properties that are of importance in the following.

\subsection{Setting of the problem and notations}\label{ref:not}
Let us introduce a few notation that will help us to set correctly the problem we are interested in.
\subsubsection{The geometric setting}\label{subsub:geom} 
Throughout this paper $\Omg\subset\R^2$ is a bounded, simply connected, $C^3$-smooth domain. The boundary of $\Omega$ is denoted by $\p\Omega$ and for $\bx\in\p\Omega$ the vector
\begin{equation*}\label{eq:normal}
	\nu(\bx) = (\nu_1(\bx),\nu_2(\bx))^\top\in\R^2
\end{equation*}
denotes the outer unit normal to $\Omg$ at the point $\bx\in\p\Omg$. We also introduce
the unit tangential vector $\tau(\bx) = (\nu_2(\bx),-\nu_1(\bx))^\top$ at $\bx\in\p\Omg$ chosen so that $\big(\tau(\bx),\nu(\bx)\big)$ is a positively-oriented orthonormal basis of $\R^2$.

We remark that the normal vector field 
$\p\Omg\ni \bx \mapsto \nu (\bx)$ induces a scalar,
complex-valued function on the boundary
\begin{equation*}\label{eq:bn}
	\bn\colon \p\Omg\arr\T,\qquad
	\bn(\bx) := 
	\nu_1(\bx) + \ii \nu_2(\bx),
\end{equation*}
where $\T := \{z\in\C\colon |z| = 1\}$.

Let $ L  > 0$ denote the length of $\p\Omg$ and consider the arc-length parametrization of $\p\Omega$ defined as $\gg \colon [0,L) \arr \R^2$ such that for all $s\in[0,L)$ we have $\gamma'(s) = \tau\big(\gamma(s)\big)$. In particular,
it means that the parametrization $\gg$ is clockwise.

Furthermore, we denote by 
\begin{equation*}\label{eq:curvature}
	\kp\colon\p\Omg\arr\R
\end{equation*}
the signed curvature of $\p\Omg$, which satisfies for all $s\in[0,L)$ the \emph{Frenet formula}
\begin{equation}\label{eq:Frenet}
	\gg''(s) = \kp\big(\gamma(s)\big) \nu\big(\gamma(s)\big). 
\end{equation}
As $\Omega$ is a $C^3$-domain, the signed curvature is a $C^1$-function on $\partial\Omega$ and we set
\begin{equation}\label{eqn:defkappastar}
	\kappa_\star := \sup_{\bx\in\partial\Omega} |\kappa(\bx)|>0,
\end{equation}
where the last inequality holds, because $\partial\Omega$ can not be a line segment.
We will also make use of the minimal radius of curvature defined by
\begin{equation}\label{eqn:defminradicurv}
	r_\rmc := \frac1{\kappa_\star}.
\end{equation}
Within our convention, the curvature of a convex domain is a non-positive function.
Finally, $\dd\Sigma$ denotes the $1$-dimensional Hausdorff measure of $\p\Omg$.
\subsubsection{Norms and function spaces}
The standard norm of a vector $\xi\in\C^n$ is defined
as $|\xi|_{\C^n}^2 := \sum_{k=1}^n|\xi_k|^2$.

The $L^2$-space and the $L^2$-based Sobolev space
of order $k \in\N$ of $\C^n$-valued functions ($n\in\N$) on the domain $\Omg$ are denoted by $L^2(\Omg,\C^n)$ and $H^k(\Omg,\C^n)$, respectively. 
The $L^2$-space and the $L^2$-based Sobolev space of order $s\in\R$
of $\C^n$-valued functions ($n\in\N$)
on the boundary $\p\Omg$ of $\Omg$ are denoted by $L^2(\p\Omg,\C^n)$ and $H^s(\p\Omg,\C^n)$, respectively. 
We use the shorthand notation $L^2(\Omg) := L^2(\Omg,\C^1)$,
$L^2(\p\Omg) := L^2(\p\Omg,\C^1)$, $H^k(\Omg) := H^k(\Omg,\C^1)$,
and $H^s(\p\Omg) := H^s(\p\Omg,\C^1)$.

We denote by $(\cdot,\cdot)_\Omg$ and by
$\|\cdot\|_\Omg$ the standard inner product
and the respective norm in $L^2(\Omg, \C^n)$.
The inner product $(\cdot,\cdot)_{\p\Omg}$ and the norm $\|\cdot\|_{\p\Omg}$ in $L^2(\p\Omg,\C^n)$
are introduced via the surface measure on $\p\Omg$.
A conventional norm in the Sobolev spaces $H^1(\Omg,\C^n)$ is defined by $\|u\|_{1,\Omg}^2 := \|\nabla u\|^2_{\Omg} + \| u\|^2_{\Omg}$.

\subsubsection{Self-adjoint operators \& the min-max principle}

Let $\sfT$ be a  self-adjoint operator 
in a Hilbert space $(\cH, (\cdot,\cdot)_\cH)$.
If $\sfT$ is, in addition, bounded from below
then let us denote by $\frt$ the associated quadratic form. 

We denote by $\speess{\sfT}$ and $\sped{\sfT}$
the essential and the discrete spectrum of $\sfT$, respectively.
By $\spe{\sfT}$, we denote the spectrum of $\sfT$ (i.e. 
$\spe{\sfT} = \speess{\sfT}\cup\sped{\sfT}$).

We say that the spectrum of $\sfT$ is \emph{discrete} if $\speess{\sfT} = \varnothing$. Let $\sfT$ be a semi-bounded operator with discrete spectrum.
For $k\in\mathbb{N}$, $\lm_k(\sfT)$ denotes the $k$-th eigenvalue of $\sfT$. These eigenvalues are ordered non-decreasingly with multiplicities taken into account. According to the min-max principle
the $k$-th eigenvalue of $\sfT$ is characterised by
\begin{equation*}\label{eq:min_max}
	\lm_k(\sfT) 
	= 
	\min_{\stackrel[{\rm dim}\,
	\cL = k]{}{\cL\subset\dom\frt}}	
	\max_{u\in\cL\setminus\{0\}}\frac{\frt[u,u]}{\|u\|_\cH^2}.
\end{equation*}
In particular, the lowest eigenvalue of $\sfT$
can be characterised as
\begin{equation}\label{eq:min_max1}
	\lm_1(\sfT) 
	= 
	\min_{u\in\dom{\frt}\setminus\{0\}}\frac{\frt[u,u]}{\|u\|_\cH^2}.
\end{equation}

\subsubsection{Pauli matrices}

Recall that the $2\tm 2$ Hermitian 
\emph{Pauli matrices} $\s_1,\s_2,\s_3$ are given by
\[
	\s_1 =  \begin{pmatrix}0&1\\1&0\end{pmatrix},
	\qquad
	\s_2  = \begin{pmatrix}0&-\ii\\ \ii&0\end{pmatrix}
	\quad \text{and} \quad 
	\s_3 = \begin{pmatrix}1&0\\0&-1\end{pmatrix}. 
\]
For $i,j\in\{1,2,3\}$, they satisfy the anti-commutation relation
\[
	\s_j\s_i + \s_i\s_j = 2 \dl_{ij},
\]
where $\dl_{ij}$ is the Kronecker symbol. 
For the sake of convenience, we define 
$\s := (\s_1,\s_2)$ and for 
$\bx = (x_1, x_2)^\top\in\R^2$ we set
\[
	\s\cdot \bx 
	:= 
	x_1 \s_1 + x_2\s_2 
	= 
	\begin{pmatrix} 
		0 & x_1 -\ii x_2\\
		x_1 + \ii x_2 & 0
	\end{pmatrix}.
\]
%

\subsection{The Dirac operator with infinite mass boundary conditions}\label{subsec:defDirope}
In this paragraph we introduce the massless Dirac operator with \emph{infinite mass} boundary conditions on $\p\Omg$, following the lines of~\cite{BFSV17a}.

\begin{dfn}\label{def:Op} 
	The massless Dirac operator with \emph{infinite mass} boundary conditions is the operator $\Op$ that acts in the Hilbert
	space $L^2(\Omg,\C^2)$ and is defined as
	\begin{equation}\label{eq:def_op}
	\begin{aligned}
		\Op u & := 
		-\ii(\s\cdot\nabla) u 
		= 
		-\ii\big(\s_1 \p_1 u + \s_2 \p_2 u\big) 
		= 
		\begin{pmatrix}
			0                &  -2\ii\p_z\\
			-2\ii\p_{\ov{z}} &  0
		\end{pmatrix} u, \\
		\dom \Op	& 
		:= 
		\big\{
		u = (u_1,u_2)^\top \in H^1(\Omg,\C^2) \colon 
					u_2|_{\p\Omg} = (\ii \bn) u_1|_{\p\Omg}  
		\big\},
		\\
	\end{aligned}		
	\end{equation}
	where
	$\p_z = \frac12\big(\p_1 - \ii \p_2\big)$
	and $\p_{\ov{z}} = \frac12\big(\p_1 + \ii\p_2\big)$ are the Cauchy-Riemann operators.
\end{dfn}
\begin{rem}\label{rem:relbengu} 
The operator $\Op$ defined in~\eqref{eq:def_op} coincides with the operator $D_\eta$ introduced 
in~\cite[\S 1.]{BFSV17a} where one chooses $\eta$ to be a constant function on the boundary $\eta := \eta(s) = \pi$. Note that we implicitly used the convention that $(\tau(\bx),\nu(\bx))$ is a positively-oriented orthonormal basis of $\R^2$ for all $\bx\in\p\Omega$.
\end{rem}

The following proposition is essentially known, we recall its proof for the sake of completeness.

\begin{prop} \label{prop:dirbasic}
	The linear operator $\Op$ defined in~\eqref{eq:def_op} satisfies the following properties.
	\begin{myenum}
		\item\label{itm:1} $\Op$ is self-adjoint.
		\item\label{itm:2} The spectrum of $\Op$ is discrete and symmetric with respect to zero.
		\item\label{itm:3} $0\notin\s(\Op)$.
	\end{myenum}
\end{prop}
\begin{proof}
\noindent (i) The self-adjointness of $\Op$ is a consequence of \cite[Theorem 1.1]{BFSV17a} where one chooses $\eta = \pi$ (see Remark~\ref{rem:relbengu}).
\smallskip

\noindent (ii) The discreteness of the spectrum 
for $\Op$ follows from compactness of the embedding $H^1(\Omega,\C^2) \hookrightarrow L^2(\Omega,\C^2)$. Regarding the symmetry of the spectrum, one can consider the charge conjugation operator
\begin{equation}\label{eqn:defchargeconj}
	C := u\in \C^2 \mapsto \sigma_1\overline{u}
\end{equation}
and notice that $\dom {\Op}$ is left invariant by $C$. Hence, a basic computation yields
\[
	\Op C = - C \Op, 	
\]
which implies that if $u\in\dom{\Op}$ is an eigenfunction
of $\Op$ associated with an eigenvalue $\mu$ then $Cu \in \dom{\Op}$ is an eigenfunction of $\Op$ associated with the eigenvalue $-\mu$, which proves the symmetry of the spectrum.
In particular the spectrum of $\Op$ consists of eigenvalues of finite multiplicity accumulating at $\pm \infty$.\smallskip

\noindent (iii)
This statement  is a consequence of \cite[Theorem 1]{BFSV17b} where we picked $\eta = \pi$;
cf. Remark \ref{rem:relbengu}.
\end{proof}
Our main interest concerns the principal eigenvalue
of $\Op$ defined as
\begin{equation*}\label{eq:pr_ev}
	\mu_\Omg = \mu_1(\Omg) := \inf\big(\spe{\Op}\cap \R_+\big) > 0.
\end{equation*}
We emphasize that the value $\mu_\Omg$ completely describes
the size of the spectral gap of $\Op$ around zero and that~\eqref{eq:main_ineq} and~\eqref{eq:main_ineq1} provide upper bounds on its length for convex and nearly circular domains, respectively.

\begin{rem}
In \cite[\S 3]{BFSV17b}, keeping the notations of \cite[\S 1.]{BFSV17a}, the massless Dirac operator with infinite mass boundary conditions is defined as a block operator $D_0 \oplus D_\pi$ and acts on $L^2(\Omega,\C^4)  = L^2(\Omg,\C^2)\oplus L^2(\Omg,\C^2)$. One easily checks that $\sigma_3 \Op \sigma_3 = - D_0$. Hence, $D_\pi = \Op$ is unitarily equivalent to $-D_0$. Thanks to the symmetry of the spectrum stated in Proposition~\ref{prop:dirbasic}\,(ii), we know that $D_0 \oplus D_\pi$ has also symmetric spectrum and that if $\mu_1(D_0 \oplus D_\pi)$ denotes the first non-negative eigenvalue of $D_0 \oplus D_\pi$ we have $\mu_1(D_0 \oplus D_\pi) = \mu_1(\Omg)$.

In addition, the authors of~\cite[\S 3]{BFSV17b},  discuss the case of the so-called \emph{armchair} boundary conditions. This operator acts in $L^2(\Omega,\C^4)$ and up to a proper unitary transform, they show that it rewrites as
\[
	{\mathsf M}_\Omega := \begin{pmatrix}0 & -\Op\\ -\Op & 0\end{pmatrix}
\]
on the domain $\dom{\Op}\oplus\dom{\Op}$. One can check that $\spe{{\mathsf M}_\Omega^2} = \spe{\Op^2}$ and thus, our results also apply to \emph{armchair} boundary conditions.
\end{rem}

Let us conclude this paragraph by mentioning the following essentially known proposition in the special case of $\Omega = \D$. For the sake of completeness, its proof is provided in Appendix~\ref{app:disk}.
\begin{prop}\label{prop:disk} 
	The principal eigenvalue $\mu_\D := \mu_1(\D)$ of $\sfD_\D$ is the smallest non-negative solution of the following scalar equation
	\[
		J_0(\mu) = J_1(\mu),
	\]
	where $J_0$ and $J_1$ are the Bessel functions of the first kind of orders $0$ and $1$, respectively. Moreover, in polar coordinates $\bx = \big(r\cos(\theta),r\sin(\theta)\big)$, an eigenfunction associated with $\mu_\D$ is
	\[
		v\big(r,\theta\big) := 
		\begin{pmatrix}
			J_0(\mu_\D r)							 \\ 
			\ii e^{\ii \theta} J_1(\mu_\D r)
		\end{pmatrix},
	\]
	where $r \in [0,1)$ and $\theta \in [0,2\pi)$.
\end{prop}

\begin{rem} An approximate numerical value of $\mu_\D$ is $\mu_\D \approx 1.434696$.
\end{rem}

\section{A variational characterization of $\mu_1(\Omg)$}\label{sec:varchar}
In this section we obtain a characterization for $\mu_\Omg = \mu_1(\Omg)$. Let us briefly outline the strategy that we follow. First,
we compute the quadratic form for the square
of the operator $\Op$. The self-adjoint 
operator $\Op^2$ is positive and its lowest eigenvalue is equal to $\mu_\Omg^2$. Therefore, it can be characterised via the min-max principle,
which gives a variational characterization of $\mu_\Omg$.
\begin{prop}\label{prop:var_form}
	The square of the principal eigenvalue
	$\mu_\Omega$ of $\Op$ can be characterised
	as
	\[
		\mu_\Omega^2 
		= 
		\inf_{u\in\dom \Op\setminus\{0\}}
		\frac{\displaystyle
		\int_\Omg|\nabla u|_{\R^2\otimes\C^2}^2\dd x 
		-
		\frac12 \int_{\p\Omg}\big(\kp |u|_{\C^2}^2\big)\dd\Sigma}
		{\displaystyle\int_\Omg |u|_{\C^2}^2\dd x}.
	\]
	In particular, $\mu_\Omg^2 \le \lm_\Omg$,
		where $\lm_\Omg$ is the lowest eigenvalue of
	 the Dirichlet Laplacian on $\Omg$.
\end{prop}

\begin{rem} 
	With the conventions chosen in \S \ref{subsub:geom}, if $\Omega$ is a convex domain we have $\kappa \leq 0$ and the boundary term in the variational characterization is non-negative. 
\end{rem}
In order to prove Proposition~\ref{prop:var_form}
we state and prove a few auxiliary lemmata.
The first lemma involves the notion of tangential derivatives. Remark that by the trace theorem~\cite[Theorem 3.37]{McL} 
there exists a constant $C = C(\Omg) > 0$ such that
\[
	\|v|_{\p\Omg}\|_{H^{3/2}(\p\Omg)} \le C\|v\|_{H^2(\Omg)}
\]
for all $v\in H^2(\Omega)$. Thus, the tangential
derivative given by
\[
	\p_\tau\colon H^2(\Omg)\arr H^{1/2}(\p\Omg),
	\qquad
	\p_\tau v := \frac{\dd}{\dd s}(v\circ \gamma),
\]
is a well-defined, continuous linear operator.
Hence, we define the tangential derivative of $u = (u_1,u_2)^\top\in H^2(\Omega,\C^2)$ by
\[
	\partial_\tau u := \big(\p_\tau u_1,\p_\tau u_2\big)^\top \in H^{1/2}(\p\Omega,\C^2).
\]
The tangential derivative is related to the square of the Dirac operator \emph{via} the next lemma,
which is reminiscent of \cite[Eq. (13)]{HMZ01}. However, we provide here a simple proof for convenience of the reader.
\begin{lem}\label{lem:lem2}
	For any $u\in H^2(\Omega,\C^2)$, one has
	\[
		\|-\ii(\sigma\cdot\nabla) u\|^2_{\Omg} 
		= 
		\|\nabla u\|^2_{\Omg} 
		- 
		\big(\ii\s_3 \p_\tau u,u\big)_{\p\Omega}.
	\]
\end{lem}

\begin{proof}
	Using an integration by parts (see~\cite[Theorem~1.5.3.1]{G85}) we get 
	for any function $v \in H^2(\Omg)$, 
	\[
	\begin{aligned}
		\int_\Omg \p_1 v \ov{\p_2 v} \dd x
		& =
		-
		\int_\Omg \ov{v}\p_{12} v \dd x
		+ \int_{\p\Omg}\big(\ov{v} \p_1 v\big) \nu_2 \dd \Sigma,\\
		\int_\Omg \ov{\p_1 v} \p_2 v \dd x
		& =
		-
		\int_\Omg \ov{v} \p_{12} v \dd x
		+ \int_{\p\Omg}\big(\ov{v}\p_2v\big)\nu_1 \dd \Sigma.
	\end{aligned}
	\]
	Dividing the difference of the above two equations
	by $2\ii$ we obtain
	\begin{equation}\label{eq:Im_parts}
	\begin{aligned}
		\Im\left(\int_\Omg \p_1 v \ov{\p_2 v} \dd x\right) & = \frac{1}{2\ii}
		\int_{\p\Omg}\ov{v}\big((\p_1 v) \nu_2
		-(\p_2 v)\nu_1\big)\dd \Sigma \\
		& =
		\frac{1}{2\ii}
		\int_{\p\Omg}\ov{v}\big(\tau\cdot\nabla v\big)
		\dd \Sigma =
		\frac{1}{2\ii}
		\int_{\p\Omg}\ov{v} \p_\tau v \dd \Sigma.
	\end{aligned}
	\end{equation}
	Let $u\in H^2(\Omg,\C^2)$.
	Using the explicit expression of $\ii(\sigma\cdot\nabla)$ and performing elementary Hilbert-space computations we get
	\begin{align*}
		\|\ii(\sigma\cdot\nabla) u\|^2_{\Omg}
		& = 
		\|\p_1 u_2 -\ii\p_2 u_2 \|^2_{\Omega}
		+
		\|\p_1 u_1 +\ii\p_2 u_1 \|^2_{\Omega}\\[0.4ex]
		& =
		\|\nabla u_1\|^2_{\Omega} + \|\nabla u_2\|^2_{\Omega}
		 + 2\Re\big[
			(\p_1 u_1,\ii\p_2 u_1)_{\p\Omg}
			-		
			(\p_1 u_2,\ii\p_2 u_2)_{\p\Omg}
		\big]	\\[0.4ex]
		&
		= \|\nabla u\|^2_{\Omg} 
		+
		2\Im\big[
		(\p_1 u_1,\p_2 u_1)_{\p\Omg}
		-
		(\p_1 u_2,\p_2 u_2)_{\p\Omg}
		\big].
	\end{align*}
	Employing identity~\eqref{eq:Im_parts} we obtain
	\[
		\|\ii(\sigma\cdot\nabla) u\|_{\Omg}^2
		=
		\|\nabla u\|^2_{\Omg}
		- 
		\big( 
		\ii \sigma_3 \p_\tau u,u
		\big)_{\p\Omg},
	\]
	which proves the claim.
\end{proof}
To obtain a convenient expression for the quadratic form of the operator $\Op^2$, we will make use of the following density lemma. 
\begin{lem}\label{lem:density} 
	$\dom{\Op}\cap H^2(\Omg,\C^2)$ is dense in $\dom \Op$ with respect to the norm $\|\cdot\|_{1,\Omg}$.
\end{lem}
\begin{proof}
	Thanks to~\cite[Theorems 1.5.1.2, 2.4.2.5,
	and Lemma 2.4.2.1]{G85} we know that there exists a bounded linear operator 
	$E \colon H^{1/2}(\p\Omega,\C^2)\rightarrow H^1(\Omega,\C^2)$ such that for any 
	$v\in H^{1/2}(\p\Omega,\C^2)$ 
	one has 
	$(Ev)|_{\p\Omega} = v$ and 
	$E\big(H^{3/2}(\Omega,\C^2)\big)\subset H^2(\Omega,\C^2)$.
	
	Let $u \in \dom{\Op}$. Since $H^{2}(\Omega,\C^2)$ is dense in $H^1(\Omega,\C^2)$ with respect to the norm $\|\cdot\|_{1,\Omg}$, there exists a one-parametric
	family of functions
	$(u_\varepsilon)_\eps \in H^2(\Omega,\C^2)$ satisfying 
	$\lim_{\varepsilon\rightarrow 0}
	\|u_\varepsilon - u\|_{1,\Omg} = 0$.
	In particular, one has
	\[
	\lim_{\varepsilon\rightarrow 0}
	\|u_\varepsilon|_{\p\Omg} - u|_{\p\Omg}\|_{H^{1/2}(\p\Omg, \C^2)} = 0.
	\]
	Now, consider the functions
	\[
	v_\varepsilon 
	:= 
	u_\varepsilon - 
	E\left(
	\frac12(1_2 + \ii\sigma_3\sigma\cdot \nu)u_\eps|_{\p\Omg}
	\right).
	\]
	Note that as defined $v_\varepsilon \in \dom{\Op}\cap H^2(\Omg,\C^2)$. Hence, we have
	\begin{multline*}
		\|u - v_\varepsilon\|_{1,\Omg} 
		\leq 
		\|u - u_\varepsilon\|_{1,\Omg} 
		+ 
		\left\|	E\left
		(\frac12(1_2+\ii\sigma_3\sigma\cdot \nu)u_\eps|_{\p\Omega}\right)\right\|_{1,\Omg}\\ 
		= \|u - u_\varepsilon\|_{1,\Omg} 
		+ 
		\left\|E\left(\frac12(1_2 + \ii\s_3\sigma\cdot \nu) (u_\varepsilon|_{\p\Omega} - u|_{\p\Omega})
		\right)\right\|_{1,\Omg},
	\end{multline*}
	where we have used that $\frac12(1_2 + \ii\sigma_3\sigma\cdot \nu) u = 0$ on $\p\Omg$ as $u\in \dom{\Op}$. Finally, using the continuity of $E \colon H^{1/2}(\p\Omega,\C^2) \to H^1(\Omega,\C^2)$ and the fact that the multiplication operator by the matrix-valued function $\p\Omg\ni x\mapsto\frac12(1_2 + \ii\sigma_3\sigma\cdot \nu)$ is bounded in $H^{1/2}(\partial\Omega,\C^2)$ we obtain that $\lim_{\varepsilon\to0} \|u - v_\varepsilon\|_{1,\Omg} = 0$ and as by definition $v_\varepsilon \in \dom{\Op}\cap H^2(\Omg,\C^2)$, we obtain the lemma.
\end{proof}
Finally, we simplify the expression
of $\|-\ii(\sigma\cdot\nabla) u\|^2_{\Omg}$ obtained in Lemma~\ref{lem:lem2} 
for the special case of functions satisfying  \emph{infinite
mass} boundary conditions.
\begin{prop}\label{prop:comput_square} 
	The identity 
	\[
		\|\Op u\|^2_{\Omg} 
		= 
		\|\nabla u\|_{\Omg}^2 - \frac12(\kp u, u)_{\p\Omg}
	\]
	holds for all $u\in\dom\Op$.
\end{prop}
\begin{proof}
	Let $u\in\dom\Op \cap H^2(\Omega,\C^2)$ be arbitrary. By Lemma~\ref{lem:lem2} we get, 
	\[
	\frb[u]
	:=
	\|\Op u\|^2_{\Omg} - \|\nabla u\|^2_{\Omg}  
	=
	-\big(
	\ii\s_3\p_\tau u, u
	\big)_{\p\Omg}\\
	=
	\ii
	\big( \p_\tau u_2,
	u_2\big)_{\p\Omg}
	-
	\ii
	\big( \p_\tau u_1, u_1\big)_{\p\Omg}.
	\]
	The boundary condition $u_2|_{\p\Omg}
	= (\ii \bn) u_1|_{\p\Omg}$ and the chain rule
	for the tangential derivative yield
	\[
	\frb[u] 	
	 =
	\ii
	\big(\bn' u_1
	+ \bn\p_\tau u_1, \bn u_1\big)_{\p\Omg}
	-
	\ii
	\big(\p_\tau u_1,
	u_1\big)_{\p\Omg}
	=
	\ii  
	\big( (\nu_1' + \ii \nu_2')u_1,\bn u_1
	\big)_{\p\Omg}.
	\]
	The Frenet formula \eqref{eq:Frenet}
	implies $\nu_2' = \kp \nu_1$ and $\nu_1' = - \kp\nu_2$.
	Plugging these identities into the above expression
	for $\frb[u]$ we arrive at
	\[
	\frb[u] =
	-
	\big(\kp (\nu_1 + \ii\nu_2)u_1
	,\bn u_1\big)_{\p\Omg}
	= - 
	\big( \kappa\bn u_1, \bn u_1
	\big)_{\p\Omg} =
	- \big(\kappa u_1,u_1\big)_{\p\Omg}\\
	= -
	\frac12(\kp u, u )_{\p\Omg},
	\]
	and the claim follows using the density of $\dom{\Op}\cap H^2(\Omg,\C^2)$ in $\dom{\Op}$ 
	with respect to the $\|\cdot\|_{1,\Omg}$-norm
	(see Lemma \ref{lem:density}).
\end{proof}
Proposition \ref{prop:comput_square}
yields the following characterization of $\mu_\D$.
\begin{corol}\label{corol:charactvpdisk} 
	The square of the principal eigenvalue
		$\mu_\D$ of $\sfD_\D$ satisfies
	\[
		\mu_\D^2 = \frac{\displaystyle\mu_\D^2 \int_0^1\left(J_0'(\mu_\D r)^2 + J_1'(\mu_\D r)^2\right)r\dd r +  \int_{0}^1\frac{J_1(\mu_\D r)^2}{r}\dd r + J_0(\mu_\D)^2}{\displaystyle\int_{0}^1\big(J_0(\mu_\D r)^2 + J_1(\mu_\D r)^2\big)r \dd r}.
	\]
\end{corol}
\begin{proof} Let $v$ be as in Proposition~\ref{prop:disk}. By definition we have $\sfD_\D v = \mu_\D v$, which implies 
\[
	\mu_\D^2 = \frac{\|\sfD_\D v\|_{\D}^2}{\|v\|^2_{\D}}.
\]
Using the representation of $\|\sfD_\D v\|_{\D}^2$
following from Proposition~\ref{prop:comput_square} and the explicit expression of $v$ in polar coordinates
given in Proposition~\ref{prop:disk}, one gets the claim.
\end{proof}


\begin{proof}[Proof of Proposition~\ref{prop:var_form}]
By Proposition~\ref{prop:comput_square} the quadratic form of $\Op^2$ is given by
\[
	\frq_\Omg[u] 
	= 
	\|\nabla u\|_\Omg^2 - \frac12(\kp u, u)_{\p\Omg},
	\qquad
	\dom {\frq_\Omg} = \dom \Op.
\]
The spectral theorem implies that $\spe{\Op^2} = \{\mu^2\colon \mu\in\spe{\Op}\}$. Hence, the lowest eigenvalue of $\Op^2$ is $\mu_\Omg^2$.
Finally, the min-max principle~\eqref{eq:min_max1} yields the sought variational characterization.
The inequality $\mu_\Omg^2\le \lm_\Omg$ follows from both
variational characterizations for $\mu_\Omg$
and $\lm_\Omg$, combined with the inclusion $H^1_0(\Omg,\C^2)\subset \dom\Op$.
\end{proof}

\section{Main result and its proof}\label{sec:mainth}

The method of the proof is inspired by a trick of G.~Szeg\H{o} presented in~\cite{S54}. His aim was to show a reversed analogue of the Faber-Krahn inequality for the first non-trivial Neumann eigenvalue in two dimensions and to do so, he used a suitably chosen conformal map 
between the unit disk and a generic simply connected domain.

Throughout this section, we identify the Euclidean plane $\R^2$ and the complex plane $\C$. Recall that $\Omg\subset\R^2$ stands for a bounded, simply connected, $C^3$-domain.

In the following, we consider a conformal map $f \colon \D \arr \Omg$. Up to a proper translation of $\Omega$ if needed and without loss of generality, we can assume that $f(0) = 0$. Remark also that $f'(z)\neq 0$ for all $z\in\D$.

As $\Omg$ is $C^3$-smooth, the Kellogg-Warschawski theorem (see~\cite[Chapter II, Theorem 4.3]{GM05} and~\cite[Theorem 3.5]{Po92}) yields that $f$ can be extended up to a function in $C^2(\ov\D)$
denoted again by $f$ with a slight abuse of notation. This extension satisfies the following natural condition $f(\T) = \p\Omg$ and the mapping
\[
	[0,2\pi)  \ni\tt \mapsto \eta(\tt) := f(e^{\ii\tt})
\] 
is a parametrization of $\p\Omg$ (see \cite[Chapter II, \S 4.]{GM05})

\subsection{A transplantation formula}
The first step in order to obtain the desired inequality is the following proposition that provides
an upper bound on the principal eigenvalue $\mu_\Omega$.
\begin{prop}
\label{prop:transplant} 
Let $\Omega\subset\R^2$ be a bounded, simply connected $C^3$-domain and let $f \colon \D \to \Omega$ be a conformal map such that $f(0) = 0$. Then one has
\begin{equation*}\label{eq:ineq_main}
	\mu_\Omega^2 \leq \frac{\cN_1 + \cN_2 + \cN_3}{\cD},
\end{equation*}
where $\cN_3 := 2\pi J_0(\mu_\D)^2$ and
\[
\begin{array}{l}
	\displaystyle\cN_1 
	:= 
	2\pi \mu_\D^2 
	\int_0^1\big(J_0'(r \mu_\D)^2 + J_1'(r \mu_\D)^2\big)r\dd r,\\[2.6ex]
	\displaystyle \cN_2 
	:= 
	\left(\int_0^1\frac{J_1(r\mu_\D)^2}r\dd r\right)
	\left(\int_0^{2\pi}\kappa\big(\eta(\theta)\big)^2 |\eta'(\theta)|^2\dd \theta\right),\\[2.6ex]
	%
	%
	\displaystyle\cD := \int_{0}^1\bigg(\big(J_0(r\mu_\D)^2 + J_1(r\mu_\D)^2\big)\int_{0}^{2\pi}|f'(re^{\ii \theta})|^2 \dd \theta\bigg) r \dd r.
\end{array}
\]
\end{prop}
\begin{proof}
First of all, note that each term $\cN_j$ (for $j=1,\dots,3$) as well as $\cD$ are well defined. In particular, the first integral appearing in $\cN_2$ is finite because
\[
	J_1(r) \sim \frac{r}2,\quad\text{when } r\to0;
\]
see \cite[Equation (10.7.3)]{OLBC10}.	

Second, observe that the composition map
\[
	V_f\colon H^1(\Omg,\C^2)\arr H^1(\D,\C^2),
	\qquad V_f u := u\circ f,
\]	
defines an isomorphism from $\dom{\sfD_\Omega}$ onto the space
\begin{equation}\label{eqn:setequalome}
	\cL_\Omega \!:=\! V_f\big(\dom{\sfD_\Omega}\big)\!  =
	\! \big\{ v = (v_1,\! v_2)^\top\! \in\! H^1(\D,\C^2)\colon\! v_2(e^{\ii \theta}) \!=\! 
	\ii\bn\big(\eta(\theta)\big)v_1(e^{\ii\theta})\big\}.
\end{equation}
Indeed, as $f$ is a conformal map, it is clear that $V_f\big(H^1(\Omega,\C^2)\big) = H^1(\D,\C^2)$. Now, let $u \in \dom{\sfD_\Omega}$. The boundary conditions read as follows
\[
\begin{aligned}
	u_2(\eta(\theta)) 
	= 
	\ii \bn\big(\eta(\theta)\big) u_1\big(\eta(\theta)\big) 
	&\quad\Longleftrightarrow\quad 
	(u_2 \circ f)(e^{\ii\theta})= \ii \bn\big(\eta(\theta)\big) (u_1 \circ f)(e^{\ii \theta})\\
	&\quad\Longleftrightarrow\quad (V_f u)_2 (e^{\ii\theta}) = \ii \bn\big(\eta(\theta)\big)(V_f u)_1(e^{\ii\theta}).
	\end{aligned}
\]
This implies the inclusion of the set on the  right-hand side of~\eqref{eqn:setequalome} into $\cL_\Omega$. The reverse inclusion is proved in the same fashion. Thus, using the variational characterization of Proposition~\ref{prop:var_form} we obtain
\begin{equation}\label{eq:var_char}
	\begin{aligned}
	\mu_\Omega^2 &= 
	\inf_{u\in\dom{\Op}\setminus\{0\}}
	\frac{{\displaystyle\int_{\Omega}}|\nabla u|_{\R^2\otimes\C^2}^2\dd x - {\displaystyle\frac12\int_{\p\Omega}}\kappa |u|_{\C^2}^2\dd\Sigma}
		{\displaystyle \int_\Omega |u|^2_{\C^2} \dd x }\\
	&= 
	\inf_{v\in\cL_\Omega\setminus\{0\}}
	\frac{{\displaystyle\int_{\D}}|\nabla v|_{\R^2\otimes\C^2}^2\dd x - {\displaystyle\frac12\int_{0}^{2\pi}}
	\kappa\big(\eta(\theta)\big)
	\big|v \big(\eta(\theta)\big)\big|_{\C^2}^2|\eta'(\theta)|\dd\theta}{
		\displaystyle\int_\D |v(x_1+\ii x_2)|_{\C^2}^2|f'(x_1+\ii x_2)|^2\dd x_1\dd x_2},
	\end{aligned}
\end{equation}
where we used that the $L^2$-norm of the gradient is invariant under conformal transformations.

Now, consider the test function $v_\star \in \cL_\Omega$ defined in polar coordinates as
\[
	v_\star\big(r,\theta\big) := 
	\begin{pmatrix}
		J_0(r\mu_\D)\\
		\ii\bn\big(\eta(\theta)\big)J_1(r\mu_\D)
	\end{pmatrix}.
\]
Plugging this test function into the 
variational characterisation~\eqref{eq:var_char}
of $\mu_\Omg^2$ we get
\[
	\mu_\Omega^2 
	\leq 
	\frac{{\displaystyle\int_{\D}}|\nabla v_\star|_{\R^2\otimes\C^2}^2\dd x 
	-
	{\displaystyle\frac12\int_{0}^{2\pi}}\kappa\big(\eta(\theta)\big) |v_\star \big(\eta(\theta)\big)|_{\C^2}^2|\eta'(\theta)|\dd\theta}
	{\displaystyle\int_\D |v_\star(x_1,x_2)|^2 |f'(x_1+\ii x_2)|^2\dd x_1 \dd x_2}.
\]
Let us compute each term in the right-hand side of the previous inequality. First, we have
\[
\begin{aligned}
	\int_\D |\nabla v_\star|_{\R^2\otimes\C^2}^2\dd x 
	&= 
	2\pi\mu_\D^2
	\int_0^1\Big(J_0'(r\mu_\D)^2 + J_1'(r\mu_\D)^2\Big)r\dd r\\&\qquad + 
	\bigg(
		\int_0^1\frac{J_1(r\mu_\D)^2}{r}\dd r
	\bigg)
	\bigg(\int_0^{2\pi} |\bn'(\eta(\theta))|^2 |\eta'(\theta)|^2
	\dd \theta\bigg)\\
	& = \cN_1 + \cN_2.
\end{aligned}
\]
Second, we obtain
\[
	-\frac12\int_0^{2\pi}
	\kappa\big(\eta(\theta)\big)
	|v_\star\big(\eta(\theta)\big)|^2_{\C^2} |\eta'(\theta)|\dd \theta 
	= 
	-J_0(r\mu_\D)^2 
	\int_{\p\Omega}\kappa\dd\Sigma 
	= 
	-2\pi J_0(r\mu_\D)^2W_\gamma,
\]
where $W_\gamma$ is the winding number of $\gamma$. 
As $\gamma$ is an arc-length
clockwise parametrization of $\p\Omega$,
we have $W_\gamma = -1$. It implies
\[
	-\frac12\int_0^{2\pi}\kappa\big(\eta(\theta)\big)|v_\star\big(\eta(\theta)\big)|^2 |\eta'(\theta)|\dd \theta 
	= \cN_3.
\]
Finally, a straightforward computation yields
\[
	\int_\D |v_\star(x_1+\ii x_2)|^2 |f'(x_1+\ii x_2)|^2\dd x_1 \dd x_2 = \cD.	\qedhere
\]
\end{proof}

\subsection{The Faber-Krahn-type inequality: rigorous statement \& proof} 

\subsubsection{Hardy spaces, conformal maps and related geometric bounds}

Recall that for any holomorphic function $g\colon\D\arr\C$ one defines its norm in the Hardy
space $\cH^2(\D)$ as follows
\[
\|g\|_{\cH^2(\D)} = \sup_{0\le r < 1}
\left(\frac{1}{2\pi}\int_0^{2\pi} |g(r e^{\ii\tt})|^2\dd \tt\right)^{1/2}.
\]
By definition, $g\in \cH^2(\D)$ means that
$\|g\|_{\cH^2(\D)} < \infty$.
If the holomorphic function $g\colon\D\arr\C$ extends up to a continuous function on $\ov{\D}$,
then $g\in \cH^2(\D)$ and 
\[
\|g\|_{\cH^2(\D)} = \left(\frac{1}{2\pi}\int_0^{2\pi} |g( e^{\ii\tt})|^2\dd \tt\right)^{1/2}.
\]
Further details on Hardy spaces can be found in~\cite[Chapter 17]{R87}.

Recall that any conformal map $f \colon \D \to \Omega$  with $f(0) = 0$ can be written as a power series
\begin{equation}\label{eq:f_expansion}
f(z) = \sum_{n\in\N} c_n z^n,
\end{equation}
for some sequence of complex coefficients $c_n\in\C$, 
$n\in\N$.

The following proposition can be found, \emph{e.g.}, in \cite[\S 3.10.2]{K16}.

\begin{prop}[Area formula]\label{prop:areaform} 
	The area of $\Omega$ is expressed through the coefficients $c_n\in\C$ of the conformal map $f$ as
	\[
	|\Omega| = \pi \sum_{n=1}^\infty n |c_n|^2.
	\]
\end{prop}

Recall that the origin
is inside $\Omg$ (i.e. $\bo\in \Omega$) and
that the radii $r_\rmi$, $r_\rmo$, and $r_\rmc$
are defined as
\begin{equation}\label{eqn:definnoutradius}
r_\rmi 
:= \min_{\bx\in\p\Omg}|\bx|,
\qquad 
r_\rmo := \max_{\bx\in\p\Omg}|\bx|,
\qquad
r_\rmc = \frac{1}{\kp_\star}.
\end{equation}
It is obvious that $r_\rmo \ge r_\rmi$
and it can also be checked that $r_\rmo \ge r_\rmc$.
In general there is no relation of this kind between
$r_\rmi$ and $r_\rmc$.

The next proposition is a consequence of the Schwarz lemma (see Koebe's estimate in \cite[Chapter I, Theorem 4.3]{GM05}).

\begin{prop}\label{prop:derradi} 
	The derivative of the conformal map $f$ at $0$ and the radius $r_\rmi$ defined in~\eqref{eqn:definnoutradius} satisfy
	\[
	|f'(0)| = |c_1| \geq r_\rmi.
	\]
\end{prop}
Next, we provide the geometric bound on $\|f'\|_{\cH^2(\D)}$ that is a consequence of~\cite[Theorem 1]{K17}. To this aim, we define for $a,b\in (0,+\infty)$ the function $\Phi$ as
\begin{equation}\label{eqn:defPhi}
	\Phi(a,b) := 
	\left\{\begin{array}{ll} 
	\frac{\ln(a) - \ln(b)}{a-b}, & \text{if } a\neq b;\\
	\frac1a,&\text{if } a=b.
	\end{array}\right.
\end{equation}
\begin{prop}[Kovalev's bound]\label{prop:kbound} 
	Let $\Omega\subset\R^2$ be a bounded, convex, $C^3$-domain and let $f \colon \D \to \Omega$ be a conformal map such that $f(0) = 0$. Then one has
	\[
	\|f'\|_{\cH^2(\D)} \le  \sup_{z\in\ov\D}|f'(z)| 
	\leq 
	r_\rmc \exp\left(2(r_\rmo - r_\rmc)\Phi(r_\rmi,r_\rmc)\right),
	\]
	with $\Phi$ defined as in~\eqref{eqn:defPhi}.
\end{prop}
\begin{rem} To recover Kovalev's bound in Proposition~\ref{prop:kbound} from~\cite[Theorem 1]{K17}, set $\lambda := (r_\rmc)^{-1} \exp\left(-2(r_\rmo - r_\rmc)\Phi(r_\rmi,r_\rmc)\right)$ and remark that for the rescaled domain $\lambda \Omega$ the radii $R_\rmi = \min_{x\in\partial(\lambda\Omega)}|x|$ and $R_\rmo = \max_{x\in\partial(\lambda\Omega)}|x|$ as well as $R_\rmc$, the minimal radius of curvature of $\lambda\Omega$, satisfy
	\[
	R_\rmi = \lambda r_\rmi,\quad R_\rmc = \lambda r_\rmc,\quad R_\rmc = \lambda r_\rmc.
	\]
	Hence, with our choice of $\lambda$ we obtain
	\[
	(R_\rmo - R_\rmc)\Phi(R_\rmi,R_\rmc) + \frac12\log(R_\rmc) = (r_\rmo - r_\rmc)\Phi(r_\rmi,r_\rmc) + \frac12\log(r_\rmc) + \frac12\log(\lambda) = 0.
	\]
	Thus, by \cite[Theorem 1]{K17}, there exists a conformal map $g \colon \D \arr \lambda\Omega$ with $g(0) = 0$ and $\sup_{z\in\overline{\D}}|g'(z)| \leq 1$. Because any conformal map from $\D$ to $\lm\Omg$ that fixes $0$ is a composition of $g$ with a rotation, any conformal map $h$ from $\D$ to $\lm\Omg$ such that $h(0) = 0$ also satisfies $\sup_{z\in\ov{\D}}|h'(z)| \leq 1$.
	
	Now, consider $h(z) := \lm f(z)$ for all $z\in \D$. As defined $h$ is a conformal map from $\D$ to $\lm\Omg$ and $h(0) = 0$. Thus, we have
	\[
	1 \geq \sup_{z\in\ov\D}|h'(z)| = \lambda \sup_{z \in \ov\D}|f'(z)|.
	\]
\end{rem}

Finally, we provide a bound on $\|f'\|_{\cH^2(\D)}$
for nearly circular domains that 
follows from~\cite[Equation 2.9]{G62} with $p = 2$.
\begin{prop}[Gaier's bound]\label{prop:gbound}
	Let $\Omega\subset\R^2$ be a bounded, $C^3$ and
	nearly circular domain in the sense of Definition~\ref{dfn:nc} with $\rho_\star \in [0,1)$.
	Let $f \colon \D \to \Omega$ be a conformal map such that $f(0) = 0$.
	Then one has
	\[
		\|f'\|_{\cH^2(\D)} \le r_\rmo
		\left(\frac{1 + \rho_\star^2}{1-\rho_\star^2}\right)^{1/2}.
	\]
\end{prop}	

\subsubsection{An abstract upper bound}

First, we formulate our main result for general
simply connected domains.
This estimate involves the norm $\|f'\|_{\cH^2(\D)}$ of the conformal map $f\colon\D\rightarrow\Omg$.
\begin{thm}\label{thm:mainthrig0} 
	Let $\Omega \subset \R^2$ be a bounded,
	simply connected, $C^3$-domain
	with $0\in\Omg$. 
	Then the following inequality holds
	\[
		\mu_1(\Omega) \leq 
		\left(\frac{2\pi}{|\Omg| + \pi r_\rmi^2}\right)^{1/2}
		\kp_\star\|f'\|_{\cH^2(\D)} \mu_1(\D),
	\]
	where $\mu_1(\Omg)$ and $\mu_1(\D)$ are  the principal 
	eigenvalues of the massless Dirac operators $\Op$
	and $\sfD_{\D}$, respectively.
	Moreover, the above inequality is strict
	unless $\Omg$ is a disk centred at the origin.
\end{thm}

\begin{proof}
	Throughout the proof we set $\mu = \mu_1(\D) > 0$ for the principal eigenvalue
	of $\OpD$. The proof relies on the analysis of each term appearing in Proposition~\ref{prop:transplant}.\\
	
	\paragraph{\textbf{The denominator $\cD$}}
	Let us start by analysing the denominator $\cD$. To do so, we will need the following claim, whose proof is postponed until the end of this paragraph.

	\noindent{\bf Claim A.}
	\emph{The function
	$r \mapsto H(r) := 
	r\left[J_0(\mu r)^2 + J_1(\mu r)^2 \right]$
	is monotonously increasing on the interval
	$(0,1)$.}

	Recall that
	\[
	\cD = 
	\int_0^1\bigg(\big(J_0(r\mu)^2 + J_1(r\mu)^2\big)\int_0^{2\pi}|f'(re^{\ii \theta})|^2 \dd \theta\bigg) r \dd r.
	\]
	Parseval's identity gives
	\[
	\int_0^{2\pi}|f'(re^{\ii \theta})|^2 \dd \theta 
	= 
	2\pi \sum_{n\in\N}n^2 |c_n|^2 r^{2n-2} 
	= 
	2\pi|c_1|^2 + 2\pi\sum_{n \geq 2}n^2|c_n|^2r^{2n-2}.
	\]
	The denominator $\cD$ rewrites as 
	\begin{equation}\label{eqn:splitD}
	\cD 
	= 
	2\pi|c_1|^2\int_0^1 H(r)\dd r 
	+ 2\pi\int_{0}^1\sum_{n\geq2}
	n^2|c_n|^2 H(r) r^{2n-2}\dd r.
	\end{equation}
	First, we handle the term
	\[
	\cI := \int_{0}^1\sum_{n\geq2}
	n^2|c_n|^2 H(r) r^{2n-2} \dd r.
	\]
	Remark that as $n^2|c_n|^2 r^{2n-2} \geq 0$ for all $r\in(0,1)$ we have
	\[
	\cI = \sum_{n\geq2}n^2|c_n|^2 
	\left(
	\int_0^1 H(r) f_n(r)\dd r\right), 
	\quad \text{with } f_n(r) = r^{2n-2}.
	\]
	Now, as for all $n\geq 2$, $f_n$ is increasing on $(0,1)$ as well as $H$ by Claim A, applying Chebyshev's inequality we get
	\[
	\begin{aligned}
	\cI 
	&\geq \left(\sum_{n\geq 2} n^2|c_n|^2 \int_0^1 f_n(r) \dd r\right)\left(\int_0^1 H(r)\dd r\right)\\ 
	&= 
	\left(\sum_{n\geq 2}\frac{n^2}{2n-1} |c_n|^2\right)\left(\int_0^1 H(r)\dd r\right)
	\geq 
	\frac12 \left(\sum_{n\geq2} n |c_n|^2\right) 
	\left(\int_0^1 H(r)\dd r\right).
	\end{aligned}
	\]
	Note that the above inequality is strict unless $c_n = 0$ for all 
	$n \ge 2$, which occurs if, and only if, $\Omg$ is a disk centred at the origin.
	Using the area formula of Proposition~\ref{prop:areaform} this inequality turns into
	\begin{equation}\label{eqn:lbI2}
	\cI \geq \frac{|\Omega| - \pi|c_1|^2}{2\pi}\int_0^1 H(r)\dd r.
	\end{equation}
	Plugging~\eqref{eqn:lbI2} into \eqref{eqn:splitD}
	and applying then Proposition~\ref{prop:derradi}, we get
	\begin{equation}\label{eqn:lbD}
	\cD \geq \left(|\Omega| + \pi |c_1|^2\right)\int_0^1 H(r)\dd r
	\geq
	\left(|\Omega| + \pi r_\rmi^2\right)\int_0^1 H(r)\dd r.
	\end{equation}
	Again we stress that the above inequality is strict unless $\Omg $ is disk centred at the origin.
	Thus, it only remains to show Claim A.
	Differentiating the function $H$
	and using the identities
	\[
		J_0'(x) = -J_1(x),\qquad 
		J_1'(x) = \frac12(J_0(x) - J_2(x)), 
		\qquad x(J_0(x) + J_2(x)) = 2J_1(x),
	\]	
	we get 
	\begin{equation*}\label{key}
		\begin{aligned}
		H'(r) 
		& =
		J_0(r\mu)^2 + J_1(r\mu)^2 
		-
		r\mu
		\left[J_0(r\mu)J_1(r\mu)   + J_1(r\mu)J_2(r\mu) \right] \\
		& 
		= J_0(r\mu)^2 + J_1(r\mu)^2 
		-
		2 J_1(r\mu)^2  =
		J_0(r\mu)^2 - J_1(r\mu)^2.
		\end{aligned}
		\end{equation*}
		Taking into account that 
		$J_0(s) > J_1(s)$ for all $s \in (0, \mu)$ we get the claim.
	
	\paragraph{\textbf{The numerator $\cN_2$}} Recall that
	\[
	\cN_2 = \left(\int_0^1\frac{J_1(r\mu)^2}{r} \dd r\right)\left(\int_0^{2\pi}\kappa^2\big(\eta(\theta)\big)|\eta'(\theta)|^2\dd \theta\right).
	\]
	By definition, for all $\theta\in (0,2\pi)$ we have $\kappa^2\big(\eta(\theta)\big) \leq \kappa_\star^2$ and moreover we get  $|\eta'(\theta)| =  |f'(e^{i\theta})|$.
	It yields
	\begin{equation}\label{eqn:ubN2}
	\cN_2 \leq 2\pi\kp_\star^2 \|f'\|_{\cH^2(\D)}^2 \left(\int_0^1\frac{J_1(r\mu)^2}{r} \dd r\right).
	\end{equation}
	\paragraph{\textbf{Combining all the estimates together.}}
	Thanks to Proposition~\ref{prop:transplant} we know that
	\[
	\big(\mu_1(\Omega)\big)^2 \leq \frac{\cN_1 + \cN_2 + \cN_3}{\cD}.
	\]
	Using~\eqref{eqn:lbD}, \eqref{eqn:ubN2} as well as the explicit expressions for $\cN_1$ and $\cN_3$ we obtain
	\[
	\big(\mu_1(\Omega)\big)^2 \leq 	
	\frac{2\pi\max\left\{1, \kp_\star^2 \|f'\|_{\cH^2(\D)}^2\right\} }{|\Omega| +\pi r_\rmi^2}\big(\mu_1(\D)\big)^2.
	\]
	Cauchy-Schwarz inequality
	and the total curvature identity yield
	\[
	\begin{aligned}
		\kp_\star^2 \|f'\|_{\cH^2(\D)}^2 
		&\ge
		\frac{1}{2\pi}
		\int_0^{2\pi} \kp(f(e^{\ii\tt}))^2|f'(e^{\ii\tt})|^2\dd \tt\\
		& \ge
		\frac{1}{4\pi^2}
		\left(
		\int_0^{2\pi} \kp(f(e^{\ii\tt}))|f'(e^{\ii\tt})|\dd \tt\right)^2
		= 1.
	\end{aligned}
	\]
	Hence, we end up with
	\[
		\big(\mu_1(\Omega)\big)^2 \leq 	
		\frac{2\pi
		\kp_\star^2\|f'\|_{\cH^2(\D)}^2}{|\Omega| +\pi r_\rmi^2}\big(\mu_1(\D)\big)^2.
	\]
	By taking the square root on both hand
	sides of the previous equation we get the claim.
	Note that the above inequality is strict unless $\Omg$
	is a disk centred at the origin.
\end{proof}
\subsubsection{Bounds for convex and for nearly circular domains}
Now we use available estimates on 
$\|f'\|_{\cH^2(\D)}$ to derive geometric	
bounds on $\mu_1(\Omg)$.	
First, we define the functional $\cF_\rmc$ that appears in~\eqref{eq:main_ineq}: 
\begin{equation}\label{eqn:functionalmainth}
	\cF_\rmc(\Omega) := 
	\left(
	\frac{|\Omega| + \pi r_\rmi^2}{2\pi}	
	\right)^{\frac12}
	\exp\big(-2(r_\rmo - r_\rmc) 
	\Phi(r_\rmi,r_\rmc)\big),
\end{equation}
where $\Phi$ is as in~\eqref{eqn:defPhi}
and the radii $r_\rmi$, $r_\rmo$ and $r_\rmc$ are given in~\eqref{eqn:definnoutradius}.
In particular, when $r_\rmi \neq r_\rmc$ the functional $\cF_\rmc$ simply rewrites as
\[
	\cF_\rmc(\Omega) = \left(
	\frac{|\Omega| + \pi r_\rmi^2}{2\pi}	
	\right)^{\frac12} \left(\frac{r_\rmc}{r_\rmi}\right)^{2 \frac{r_\rmo - r_\rmc}{r_\rmi - r_\rmc}}.
\]
Remark that $\Phi(a,b) \ge 0$ for any $a,b\in\R_+$. Furthermore, the functional $\cF_\rmc$  has the following 
properties.
\begin{enumerate}[label=(\alph*)] 
	\item\label{itm:pteF1} For any $\Omega$ and all $\alpha > 0$ one has $\cF_\rmc(\alpha \Omega) = \alpha \cF_\rmc(\Omega)$.
	\item\label{itm:pteF2} 
	One has for any $\Omega$
	\[
	\cF_\rmc(\Omega)\le 	\left(
	\frac{|\Omega| + \pi r_\rmi^2}{2\pi}	
	\right)^{\frac12}\le \sqrt{\frac{|\Omg|}{\pi}}
	\]
	and, in particular, $\cF_\rmc(\D_r) =  \sqrt{\frac{|\D_r|}{\pi}} = r$.
	\item\label{itm:pteF3} 
	$\cF_\rmc$ is not invariant under translations. Indeed, for $\Omega = \D$, we have $r_\rmo = r_\rmi = r_\rmc = 1$ and $\cF_\rmc(\Omega) = 1$. 
	However, if one picks $\Omega = \D + (\frac12,0)$,
	then one has
	$r_\rmi = \frac12$, $r_\rmo = \frac32$,
	$r_\rmc = 1$
	and
	\[
		\cF_\rmc(\Omega) = 
		\frac18\left(\frac{5}{2}\right)^{\frac12}
		\simeq 0.198.
	\]
\end{enumerate}

Now, we have all the tools to rigorously formulate our main result for convex domains. 
This result is just a simple consequence of
Theorem~\ref{thm:mainthrig0}, in which $\|f'\|_{\cH^2(\D)}$ is estimated via Proposition~\ref{prop:kbound} and the scaling property
$r\mu_1(\D_r) = \mu_1(\D)$ is employed.
\begin{thm}\label{thm:mainthrig} 
Let $\Omega \subset \R^2$ be a bounded, convex, $C^3$-domain 
such that $\bo\in\Omg$ and let the functional $\cF_\rmc(\cdot)$
be as in~\eqref{eqn:functionalmainth}.
Then the following inequality holds
\[
	\cF_\rmc(\Omega)\mu_1(\Omega) \leq \cF_\rmc(\D_r)\mu_1(\D_r),
\]
where
$\mu_1(\Omg)$ and $\mu_1(\D_r)$ are  the principal
eigenvalues of the massless Dirac operators $\Op$
and $\sfD_{\D_r}$, $r>0$, respectively.
Moreover, the above inequality is strict
unless $\Omg$ is a disk centred at the origin.
\end{thm}

\begin{rem}
Condition~\ref{itm:pteF1} implies that the family
\[	
	\cE_\rmc(r) := \big\{\Omega\text{ is a bounded, convex  $C^3$-domain} \colon \cF_\rmc(\Omega) = r\big\}, \quad r >0.
\] 
is non-empty and contains ``many'' domains.
\end{rem}

\begin{corol}
Let the assumptions be as in Theorem~\ref{thm:mainthrig}.
Then the following inequality
\[
	\mu_1(\Omega) < \mu_1(\D_r)
\]
holds provided that $\cF_\rmc(\Omega) = r$
and that $\Omg\ne\D_r$.
\end{corol}

Note that thanks to property~\ref{itm:pteF3} we know that $\cF_\rmc$ is sensitive to the choice of the origin. Stated as it is, Theorem \ref{thm:mainthrig} can still be slightly optimized, because the principal eigenvalue
itself is clearly insensitive to translations
of $\Omg$.
Thanks to~\ref{itm:pteF2}, we have 
\[
	\cF_\rmc(\Omega -y) \leq \sqrt{\frac{|\Omega -y|}{\pi}} = \sqrt{\frac{|\Omega|}{\pi}}
\] 
and hence Theorem~\ref{thm:mainthrig} immediately yields the following corollary.  

\begin{corol}\label{cor:FKopt}
Let the assumptions be as in Theorem~\ref{thm:mainthrig}. 
Then the following inequality holds
\[
	\mu_1(\Omega) \leq \frac{r}{\cF_\rmc^\star(\Omg)}\mu_1(\D_r),
\]
where $\cF_\rmc^\star(\Omega) := 
\sup_{y\in\Omega} \cF_\rmc(\Omega - y)$.
\end{corol}

Stated this way, the upper bound in the right hand side of the inequality in Corollary \ref{cor:FKopt} is translation invariant. However, the upper bound is no longer expressed in term of simple geometric quantities. Nevertheless, if the domain $\Omega$ has some extra symmetries, one can find explicitly $y_\star\in\Omega$, which maximizes the function $y\mapsto \cF_\rmc(\Omega - y)$. This is the purpose of the following proposition, whose proof is postponed
to Appendix~\ref{app:prop}.

\begin{prop}\label{prop:symm} 
	Let $\Omega$ be a bounded, convex $C^3$-domain,
	which	has two axes of symmetry $\Lambda_1$ and $\Lambda_2$ that intersect in a unique point $y_\Lambda\in\Omg$, then $\cF_\rmc(\Omega - y_\Lambda) = \cF_\rmc^\star(\Omega)$.
\end{prop}

Proposition~\ref{prop:symm} immediately yields the optimal bound that one can obtain in Corollary~\ref{cor:FKopt} whenever $\Omega$ has two axes of symmetry. 
For example, let $0<b<a$ and take for $\Omega$ the ellipse of major axis $2a$ and minor axis $2b$ defined as
\[	
	\Omega := \left\{ (x_1,x_2)^\top\in\R^2 \colon \frac{x^2_1}{a^2} + \frac{x^2_2}{b^2} \leq 1\right\}.
\] 
One easily finds $r_\rmc =a^{-1} b^2$ and by Proposition~\ref{prop:symm} the optimal choice of $y\in\Omega$ to minimize $\cF(\Omega -y)$ is given by $y= \bo$.
Hence, $r_\rmi = b$ and $r_\rmo = a$ and we obtain
\[
	\cF_\rmc^\star(\Omg) = \sup_{y\in\Omega}\cF_\rmc(\Omega - y) 
	= 
	\cF_\rmc(\Omega) 
	= 
	\left(\frac{ab+b^2}{2}\right)^{\frac12} \exp\left(-2\Phi(b,a^{-1}b^2)\frac{a^2-b^2}{a}\right).
\]
Remark that as $a>b>0$ we have $\cF_\rmc^\star(\Omega) < \sqrt{ab}$. 
\begin{rem}	
We also observe that for the ellipse 
$\Omg_x\subset\R^2$ centred at the origin
with $a = 1 + x$ and $b = \frac{1}{1+x}$
for some $x > 0$
one has
\[
	\cF_\rmc(\Omg_x) = \left(\frac{2+2x+x^2}{2+4x+2x^2}\right)^{1/2}
	\left(\frac{1}{1+x}\right)^{8+8x+4x^2} = 1 - \frac{17}{2}x + O(x^2),\qquad x\arr 0^+.
\]
Thus, the upper bound in Theorem~\ref{thm:mainthrig} is reasonably precise if $x > 0$ is small, in which
case the ellipse $\Omg_x$ is close to the unit disk.
On the other hand, $\cF_\rmc(\Omg_x)$ decays super-exponentially for $x \arr \infty$ and in that
regime  the obtained upper bound on $\mu_1(\Omg)$ is very rough.
\end{rem}

In what follows we assume that $\Omg$ is a nearly circular domain in the sense of Definition~\ref{dfn:nc}.
Now, we define the functional that 
appears in~\eqref{eq:main_ineq1}:
\begin{equation}\label{eqn:functionalmainth2}
	\cF_\rms(\Omega) := 
	\left(
	\frac{|\Omega| + \pi r_\rmi^2}{2\pi}	
	\right)^{\frac12}
	\frac{r_\rmc}{r_\rmo}
	\left(\frac{1-\rho_\star}{1+\rho_\star}\right)^{1/2}.
\end{equation}
The functional $\cF_\rms$  shares common properties with $\cF_\rmc$.
\begin{enumerate}[label=(\alph*)] 
	\item\label{itm:pteFs1} 
	For any nearly circular $\Omega$ and all $\alpha > 0$ one has $\cF_\rms(\alpha \Omega) = \alpha \cF_\rms(\Omega)$.
	\item\label{itm:pteFs2} 
	One has for any nearly circular $\Omega$
	\[
	\cF_\rms(\Omega)\le 	\left(
	\frac{|\Omega| + \pi r_\rmi^2}{2\pi}	
	\right)^{\frac12}\le \sqrt{\frac{|\Omg|}{\pi}}
	\]
	and, in particular, $\cF_\rms(\D_r) =  \sqrt{\frac{|\D_r|}{\pi}} = r$.
	\item\label{itm:pteFs3} 
	$\cF_\rms$ is also not invariant under translations.
\end{enumerate}

Now, we have all the tools to rigorously formulate our main result for nearly circular domains. 
	This result is also a simple consequence of
	Theorem~\ref{thm:mainthrig0}, in which $\|f'\|_{\cH^2(\D)}$ is now estimated via Proposition~\ref{prop:gbound}.
\begin{thm}\label{thm:mainthrig2} 
	Let $\Omega\subset\R^2$ be a bounded, $C^3$ and
	nearly circular domain in the sense of Definition~\ref{dfn:nc} with $\rho_\star \in [0,1)$ and let the functional $\cF_\rms(\cdot)$
	be as in~\eqref{eqn:functionalmainth2}.
	Then the following inequality holds
	\[
	\cF_\rms(\Omega)\mu_1(\Omega) \leq \cF_\rms(\D_r)\mu_1(\D_r),
	\]
	where
	$\mu_1(\Omg)$ and $\mu_1(\D_r)$ are  the principal
	eigenvalues of the massless Dirac operators $\Op$
	and $\sfD_{\D_r}$, $r>0$, respectively.
	Moreover, the above inequality is strict
	unless $\Omg$ is a disk centred at the origin.
\end{thm}

\begin{rem}
Condition~\ref{itm:pteFs1} implies that the family
\[	
	\cE_\rms(r) := \big\{\Omega\text{ is a bounded, nearly circular  $C^3$-domain} \colon \cF_\rms(\Omega) = r\big\}, \quad r >0.
\] 
is non-empty and contains ``many'' domains.
\end{rem}

\begin{corol}
Let the assumptions be as in Theorem~\ref{thm:mainthrig2}.
Then the following inequality
\[
	\mu_1(\Omega) < \mu_1(\D_r)
\]
holds provided that $\cF_\rms(\Omega) = r$
and that $\Omg\ne\D_r$.
\end{corol}

\section*{acknowledgements}
The authors are very grateful to Lo\"ic Le Treust, Konstantin Pankrashkin and Leonid Kovalev for fruitful discussions.

VL acknowledges the support by the grant No. 17-01706S of the Czech Science Foundation (GA\v{C}R) and by a public grant as part of the Investissement d'avenir project, reference ANR-11-LABX-0056-LMH, LabEx LMH. A large part of the work was done during two stays of VL at the University Paris-Sud in 2018.

TOB is supported by the ANR "D\'efi des autres savoirs (DS10) 2017" programm, reference ANR-17-CE29-0004, project molQED and by the PHC Barrande 40614XA funded by the French Ministry of Foreign Affairs and the French Ministry of Higher Education, Research and Innovation.
TOB is grateful for the stimulating research stay and the hospitality of the Nuclear Physics Institute of Czech Republic where this project has been initiated.

\appendix

\section{The massless Dirac operator with infinite mass boundary conditions on a disk}\label{app:disk}
The goal of this appendix is to prove Proposition \ref{prop:disk}. Namely, we are aiming to characterize
the principal eigenvalue $\mu_\D$ and the associated eigenfunctions for the self-adjoint operator $\OpD$ on the unit disk 
\[
	\D = \{\bx\in\R^2\colon |\bx| < 1\}.
\]
The material of this appendix is essentially known (see for instance~\cite[App. D]{V17}).
However, we recall it here for the sake of completeness.

\subsection{The representation of the operator $\OpD$ in polar coordinates} 
First, we introduce the polar coordinates $(r,\tt)$
on the disk $\D$. They are related to the Cartesian
coordinates $\bx = (x_1,x_2)$ \emph{via} the identities
\[
	\bx(r,\tt) = 
	\begin{pmatrix}
	x_1(r,\tt)\\
	x_2(r,\tt)
	\end{pmatrix},\quad\text{where}\quad x_1 = x_1(r,\tt) = r\cos\tt,\quad
	x_2 = x_2(r,\tt) = r\sin\tt,
\]
for all $r \in \I := (0,1)$ and $\tt\in\T$.
Further, we consider the moving frame 
$(\be_{\rm rad}, \be_{\rm ang})$ associated with the polar coordinates
\[
	\be_{\rm rad}(\tt) 
	=
	\frac{\dd \bx}{\dd r}
	= 
	\begin{pmatrix}
		\cos\tt\\
		\sin\tt
	\end{pmatrix}
	\quad\text{and}\quad
	\be_{\rm ang}(\tt) 
	= 
	\frac{\dd \be_{\rm rad}}{\dd\tt}
	 = 	
	\begin{pmatrix}
	-\sin\tt\\
	\cos\tt
	\end{pmatrix}.
\]
The Hilbert space $L^2_\cyl(\D,\C^2) := L^2(\I\times\T,\C^2;r\dd r \dd \theta)$ can be viewed
as the tensor product $L^2_r(\I)\otimes L^2(\T,\C^2)$, where 
$L^2_r(\I) = L^2(\I;r\dd r)$.
Let us consider the unitary transform
\[
	V \colon L^2(\D,\C^2) \arr L^2_\cyl(\D,\C^2),\qquad 
	(Vv)(r,\theta) = u\big(r\cos\theta,r\sin\theta\big),
\]
and introduce the cylindrical Sobolev space by
\[
	H^1_\cyl(\D,\C^2) := V\big(H^1(\D,\C^2)\big) = \Big\{v \in L^2_\cyl(\D,\C^2) \colon \p_r v, r^{-1}(\p_\theta v) \in L^2_\cyl(\D,\C^2)\Big\}
\]
We consider the operator acting in the Hilbert space $L^2_\cyl(\D,\C^2)$ defined as
\begin{equation}\label{eqn:unitequivD}
	\wt\OpD := V \OpD V^{-1},\quad \dom{\wt\OpD} = V\big(\dom{\OpD}\big).
\end{equation}
Now, let us compute the action of $\wt\OpD$ on a function $v\in \dom{\wt\OpD}$. Notice that there exists $u\in \dom{\OpD}$ such that $v = Vu$ and the partial derivatives of $v$ with respect to the polar variables $(r,\tt)$ can be expressed through those of $u$ with respect to the Cartesian variables $(x_1,x_2)$ via the standard relations (for $\bx = \bx(r,\tt)$)
\[
\begin{aligned}
	(\p_r v)(r,\tt) &  =
	\sin\tt(\p_2 u)(\bx) + \cos\tt(\p_1 u)(\bx),
	\\
	r^{-1}(\p_\tt v)(r,\tt) 
	& =  
	\cos\tt(\p_2 u)(\bx) - \sin\tt(\p_1 u)(\bx),
\end{aligned}
\]
and the other way round
\[
\begin{aligned}
	(\p_1 u)(\bx) & = 
	\cos\tt (\p_r v)(r,\tt) - \sin\tt\frac{(\p_\tt v)(r,\tt)}{r},\\
	(\p_2 u)(\bx) &=
	\sin\tt (\p_r v)(r,\tt) + \cos\tt\frac{(\p_\tt v)(r,\tt)}{r}.
\end{aligned}
\]
Using the latter formul{\ae} we can express the action
of the differential expression $-\ii(\s\cdot\nabla)$ in polar coordinates as follows
(for $\bx = \bx(r,\tt)$)
\[
\begin{aligned}
	(-\ii(\s\cdot\nabla) u)(x) 
	& = 
	-\ii\begin{pmatrix} 
		\p_1 u_2(\bx) - \ii \p_2 u_2(\bx)\\
		\p_1 u_1(\bx) + \ii \p_2 u_1(\bx)
	\end{pmatrix}\\
	& =
	-\ii\begin{pmatrix} 
		e^{-\ii\tt} (\p_r v_2)(r,\tt) -
		\ii e^{-\ii\tt}r^{-1}(\p_\tt v_2)(r,\tt) 
		\\
		e^{\ii\tt} (\p_r v_1)(r,\tt) 
		+ \ii e^{\ii\tt}r^{-1}(\p_\tt v_1)(r,\tt)
	\end{pmatrix}.
\end{aligned}                        
\]
Note that a basic computation yields
\begin{equation}\label{eq:matrix_identity}
	\s\cdot \be_{\rm rad} 
	= 
	\cos\tt\s_1 + \sin\tt\s_2 
	= 
	\begin{pmatrix}
		0           & e^{-\ii\tt}\\
		e^{\ii\tt}  &0
	\end{pmatrix}.
\end{equation}
Hence, the operator
$\wt\OpD$ acts as
\begin{equation}\label{eqn:expopcyl}
\begin{aligned}
	\wt\OpD v & = 
	-\ii(\s\cdot \be_{\rm rad})
	\left(\p_r v + \frac{v - \s_3 \sfK v}{2r}\right),\\
	\dom{\wt\OpD} & = 
	\big\{ v\in H_{\rm cyl}^1(\D,\C^2)\colon v_2(1,\tt) = 
	\ii e^{\ii\tt} v_1(1,\tt)\big\},
\end{aligned}
\end{equation}
where $\sfK$ is the spin-orbit operator in the Hilbert space $L^2(\T;\C^2)$ defined as
\begin{equation}\label{eq:K}
	\sfK = -2\ii\p_\tt + \s_3,
	\qquad 
	\dom\sfK = H^1(\T,\C^2).
\end{equation}
Let us investigate the spectral properties of the spin-orbit operator $\sfK$.
\begin{prop}\label{prop:K} 
	Let the operator $\sfK$ be as in~\eqref{eq:K}.
	Then the following hold.
	\begin{myenum}
		\item\label{itm:K1} $\sfK$ is self-adjoint and has a compact resolvent.
		\item\label{itm:K2} $\spe{\sfK} = \{2k+1\}_{k\in\Z}$ 
		and $\cF_k := \ker\big(\sfK - (2k+1)\big) = \spn(\phi_k^+, \phi_k^-)$, where 
		\[
			\phi_k^+ = \frac1{\sqrt{2\pi}}
			\begin{pmatrix}
				e^{\ii k\tt}\\0
			\end{pmatrix}
			\quad\text{and}\quad 
			\phi_k^- 
			= \frac1{\sqrt{2\pi}}
			\begin{pmatrix}
			0\\ e^{\ii (k+1)\tt}
			\end{pmatrix}.
		\]
		\item\label{itm:K3} $(\s\cdot \be_{\rm rad}) \phi_k^\pm = \phi_k^\mp$
		and $\s_3 \phi_k^\pm = \pm \phi_k^\pm$.
\end{myenum}
\end{prop}
\begin{proof}
	\noindent (i)
	The operator $\sfK$ is clearly self-adjoint in $L^2(\T,\C^2)$, because adding the matrix $\s_3$ 
	can be viewed as a symmetric bounded perturbation of an unbounded self-adjoint momentum operator $H^1(\T,\C^2)\ni \phi \mapsto -\ii \phi'$ in the 
	Hilbert space $L^2(\T,\C^2)$.
	As $\dom\sfK = H^1(\T,\C^2)$ is compactly embedded into $L^2(\T,\C^2)$ the resolvent of $\sfK$ is compact.
	\medskip
	
	\noindent (ii)
	Let $\phi = (\phi^+,\phi^-)^\top\in\dom\sfK$ and $\lambda\in\R$ be such that $\sfK \phi =\lambda \phi$.
	The eigenvalue equation on $\phi$ reads as follows
	\[
		(\phi^\pm)'  =  \frac{\ii}{2}
		\big(\lambda \mp 1\big)\phi^\pm.
	\]
	The generic solution of the above system of differential equations is given by
	\[
		\phi^\pm(\tt) = A_\pm \exp\big(\ii \tfrac{\lambda\mp 1}2 \tt\big), \qquad A_\pm\in\C.
	\]	
	Hence, the periodic boundary condition $\phi^\pm(0) = \phi^\pm(2\pi)$ implies that the eigenvalues of 
	$\sfK$ are exhausted by  
	$\lambda = 2k + 1$ for $k\in\Z$ and that $\{\phi_{k}^+,\phi_{k}^-\}$ is a basis of $\cF_k$.
	\medskip
	
	\noindent (iii) 
	These algebraic relations  are obtained {\it via} basic matrix calculus using~\eqref{eq:matrix_identity}.
\end{proof}
We are now ready to introduce subspaces of $\dom{\wt\OpD}$ that are invariant under its action.
The analysis of $\wt\OpD$ reduces to the study of its restrictions to each invariant subspace.
\begin{prop}
	There holds
	\[
		L^2_\cyl(\D, \C^2\big)\simeq 
		L^2_r(\I)\otimes L^2(\T,\C^2) 
		= 
		\oplus_{k\in\Z} \cE_k,
	\]
	where $\cE_k = L^2_r(\I)\otimes \cF_k$ and $L^2_r(\I):= L^2(\I;r\dd r)$. Moreover, the following hold true.
	\begin{myenum}
	\item\label{itm:rs1} For any $k\in\Z$,   
	\[
		d_k u := \wt\OpD u,\qquad
		\dom{d_k} := \dom{\wt\OpD} \cap \cE_k
	\]
	is a well-defined self-adjoint operator in the Hilbert space $\cE_k$. 
	\item\label{itm:rs2} For any $k \in \Z$, the operator $d_k$ is unitarily equivalent to the operator $\bd_k$ in the Hilbert space $L^2_r(\I,\C^2)$ defined as
	\begin{equation}\label{eq:op_dk}
	\begin{aligned}
		\bd_k& \! =\!
		\begin{pmatrix}
			0 	& -\ii \frac{\dd}{\dd r} - \ii\frac{k+1}{r}\\
			-\ii\frac{\dd}{\dd r} + \ii \frac{k}r	&0
		\end{pmatrix},\\ 
		\dom{\bd_k} &\! =\!
		\left\{u =(u_+,u_-) \colon
		u_\pm, u_\pm', \tfrac{ku_+}{r}, \tfrac{(k+1)u_-}{r} \in L^2_r(\I), 
		u_-(1) = \ii u_+(1)\right\}.
	\end{aligned}
	\end{equation}
	\item\label{itm:rs3}
	$\spe{\OpD} = \spe{\wt\OpD} = \bigcup_{k\in\Z}\spe{\bd_k}$.
\end{myenum}
\label{prop:fibdecdisk}
\end{prop}
\begin{proof}
\noindent (i) Let us check that $d_k$ is well defined. Pick a function $u\in\dom{\wt\OpD}\cap \mathcal{E}_k$. By definition, $u$ writes as
\[
	u(r,\theta) = u_+(r)\phi_k^+(\theta) + u_-(r)\phi_k^-(\theta),
\]
and, since $u\in H_{\rm cyl}^1(\D,\C^2)$, we have $u_\pm,u_\pm',\frac{k}ru_+,\frac{k+1}r u_-\in L_r^2(\I)$. Applying the differential expression obtained in \eqref{eqn:expopcyl}, we get
\begin{equation}\label{eqn:expropocyl}
	\begin{aligned}
	(\wt \OpD u)(r,\theta) 	&\!=\! -\ii(\s\cdot \be_{\rm rad})
	\left(\p_r v + \frac{v - \s_3 \sfK v}{2r}\right)u(r,\theta)\\
				&=\! \left[-\ii u'_-(r)\! -\! \frac{\ii(k+1)}{r} u_-(r)\right]\!\phi_k^+(\theta)\! +\! \left[-\ii u_+'(r) \!+\!  \frac{\ii k}r u_+(r)\right]\!\phi_k^-(\theta).
	\end{aligned}
\end{equation}
It yields $\wt \OpD \left(\dom{\OpD}\cap\cE_k\right) \subset \cE_k$. It is now an easy exercise to show that $d_k$ is self-adjoint.
\medskip

\noindent(ii) Let us introduce the unitary transform
\[
	W_k \colon  \cE_k \arr L_r^2(\I,\C^2),\quad (W_ku)(r) = \Big((u(r,\cdot),\phi_k^+)_{L^2(\T,\C^2)}, 
		(u(r,\cdot),\phi_k^-)_{L^2(\T,\C^2)}\Big)^\top.
\]
For $u\in \cE_k$ it is clear that we have $\|W_ku\|_{L_r^2(\I,\C^2)} = \|u\|_{L^2_\cyl(\D,\C^2)}$ and we observe that 
\[
	\bd_k = W_kd_k W_k^{-1},\qquad
	\dom{\bd_k} = W_k\left(\dom{d_k}\right).
\]
%
\medskip

\noindent (iii) The first equality is a consequence of \eqref{eqn:unitequivD}, while the second one is an application of \cite[Theorem XIII.85]{RS78}.
\end{proof}

\subsection{Eigenstructure of the disk}
%
Before describing the eigenstructure of the disk recall that $C$ denotes the charge conjugation operator introduced in \eqref{eqn:defchargeconj}. It is not difficult to see that $C$ is anti-unitary and maps $\dom{\bd_k}$ onto $\dom{\bd_{-(k+1)}}$ for all $k\in\Z$. Furthermore, a computation yields
\begin{equation}\label{eqn:conjdkC}
	C \bd_{-(k+1)}C = - \bd_k.
\end{equation}
In particular, $C^2 = 1_2$, which also reads $C^{-1} = C$. Combined with \eqref{eqn:conjdkC} and as the spectrum of $\bd_k$ is discrete one immediately observes that
\begin{equation}	\label{eqn:symspc}
	\spe{\bd_k} = - \spe{\bd_{-(k+1)}}.
\end{equation}
Hence, we can restrict ourselves to $k\geq0$.

\begin{lem}\label{lem:ineqfq}
	Let $k\in\N_0$. Let $\bd_k$ be the self-adjoint operator defined in~\eqref{eq:op_dk}.	Then for all $k\in\N$
	the following hold.
	\begin{myenum}
		\item $\dom{\bd_k} \subset \dom{\bd_0}$
		\item 
		$\|\bd_k u\|_{L^2_r(\I;\C^2)}^2 \geq 
		\|\bd_0 u\|_{L^2_r(\I;\C^2)}^2$
		for all $u \in \dom{\bd_k}$.	
	\end{myenum}	
\end{lem}

\begin{proof}
	Let $k\in \N$ and 
	$u= (u_+,u_-)^\top\in\dom{\bd_k}$. It is clear that $u\in\dom{\bd_0}$ and that for integrability reasons $u(0) = 0$. Hence, we have
	\begin{equation}\label{eq:1}
	\begin{aligned}
		\left\|u_+' - \tfrac{k}r u_+\right\|_{L^2_r(\I)}^2 & 
		= 
		\left\|u_+'\right\|_{L^2_r(\I)}^2 -
		2k\Re\left(u_+',\tfrac{1}{r}u_+\right)_{L^2_r(\I)}
		+ 
		k^2
		\left\| \tfrac{1}{r} u_+\right\|_{L^2_r(\I)}^2\\
		& \ge
		\left\|u_+'\right\|_{L^2_r(\I)}^2 -
		2k\Re\int_0^1 u_+'\ov{u_+}\dd r\\
		& =
		\left\|u_+'\right\|_{L^2_r(\I)}^2 
		-
		k\int_0^1(|u_+|^2)'\dd r\\
		& =
		\left\|u_+'\right\|_{L^2_r(\I)}^2 
		-
		k|u_+(1)|^2.
	\end{aligned}
	\end{equation}
	Analogously, we get 
	\begin{equation}\label{eq:2}
	\begin{aligned}
		\left\|u_-' + \tfrac{k+1}r u_-\right\|_{L^2_r(\I)}^2 & 
		\ge
		\left\|
			u_-' + \tfrac{1}r u_-
		\right\|_{L^2_r(\I)}^2 
		+
		2k\Re\left(u_-',\tfrac{1}{r}u_-\right)_{L^2_r(\I)}\\
		& =
		\left\|
		u_-' + \tfrac{1}r u_-
		\right\|_{L^2_r(\I)}^2 
		+
		k\int_0^1(|u_-|^2)'\dd r\\
		& 
		=
		\left\|
		u_-' + \tfrac{1}r u_-
		\right\|_{L^2_r(\I)}^2 
		+
		k|u_-(1)|^2
	\end{aligned}
	\end{equation}
	Combining~\eqref{eq:1} and~\eqref{eq:2}
	with the boundary condition $u_-(1) = \ii u_+(1)$ we get
	\[
	\begin{aligned}
		\|\bd_ku\|^2_{L^2_r(\I,\C^2)} 
		& \ge
		\left\|
			u_+'\right\|_{L^2_r(\I)}^2 
		+
		\left\|
		u_-' + \tfrac{1}r u_-
		\right\|_{L^2_r(\I)}^2 
		+ k\big(|u_-(1)|^2 - |u_+(1)|^2\big)\\
		& = 
		\|\bd_0u\|_{L^2_r(\I,\C^2)}^2
		+ k\big(|u_-(1)|^2 - |u_+(1)|^2\big)
		= \|\bd_0 u\|_{L^2_r(\I,\C^2)}^2.\qedhere
	\end{aligned}
	\]
\end{proof}

Now, we have all the tools to prove Proposition \ref{prop:disk}.

\begin{proof}[Proof of Proposition \ref{prop:disk}]
	As a direct consequence of Lemma~\ref{lem:ineqfq} and the min-max principle, we obtain that
	\[
		\mu_1(\bd_k^2) \geq \mu_1(\bd_0^2) 
		= \mu_\D^2.
	\]
	Thus, by Proposition \ref{prop:fibdecdisk}\,(iii) and Equation~\eqref{eqn:symspc}, in order to investigate the first eigenvalue of $\OpD$, we only have to focus on the operator $\bd_0$.
	
	 Let $\mu > 0$ be an eigenvalue of $\bd_0$ and $u$ be an associated eigenfunction. 
	 In particular, $u = (u_+,u_-)^\top \in \dom{\bd_0^2}$ and we have
	 \[
		0 = (\bd_0 + \mu)(\bd_0 - \mu) u = \begin{pmatrix}
		-u_+'' - \frac{u_+'}{r} - \mu^2 \\ -u_-'' - \frac{u_-'}{r} + \frac{u_-}{r^2} - \mu^2\end{pmatrix} .
	\]
	Hence, we obtain
	\[
		u_+(r) 
		= 
		a_+J_0(\mu r) + b_+Y_{0}(\mu r)
		\quad\text{and}\quad 
		u_-(r) = 
		a_- J_1(\mu r) + b_- Y_1(\mu r),
	\]
	with some constants $a_\pm,b_\pm \in\C$ and where $J_\nu$ and $Y_\nu$ ($\nu =0,1$) denote the Bessel function of the first kind of order $\nu$ and the Bessel function of the second kind of order $\nu$, respectively.
	Taking into account that
	\[
		\lim_{r\arr 0^+} r^2 |Y_0'(r)|^2 {=}  
		 \frac{4}{\pi^2},\quad \lim_{r\arr 0^+}r^4 |Y_1'(r)|^2 {=}\frac{4}{\pi^2},
	\]
  (see \cite[\S 10.7(i)]{OLBC10}), the condition $u\in \dom{\bd_0}$ implies $b_\pm = 0$ or, in other words,
   \begin{equation*}\label{key}
    	u_+(r) = a_+ J_0(\mu r)  	
   	\quad\text{and}\quad 
   	u_-(r) = a_- J_1(\mu r). 
   \end{equation*} 
	Now, as $u$ satisfies the eigenvalue equation $\bd_0u = \mu u$ we get 
	$u_+' = \ii\mu u_-$ and 
	the identity
	\[
		a_-\mu J_1(\mu r) = \ii \mu a_+ J_1(\mu r)
	\]
	holds for all $r\in \I$. 
	In particular, we obtain $a_- = \ii a_+$ which gives
	\begin{equation}\label{eq:phi}
		u = a_+
		\begin{pmatrix}
			J_0(\mu r)\\	\ii J_1(\mu r)
		\end{pmatrix}.
	\end{equation}
	Now, the boundary condition $u_-(1) = \ii u_+(1)$ reads as
	\begin{equation}\label{eq:secular}
		J_0(\mu) = J_1(\mu),
	\end{equation}
	which gives the eigenvalue equation, whose first
	positive root is the principal eigenvalue of $\OpD$.
	An eigenfunction of $\wt\OpD$ corresponding
	to the eigenvalue $\mu_\D$ is given in polar
	coordinates by
	\[
		w(r,\theta) = u\otimes \phi_0^- = \frac1{\sqrt{2\pi}}\begin{pmatrix}J_0(r\mu_\D )\\\ii e^{\ii \theta}J_1(r\mu_\D )\end{pmatrix}
	\]
	where $\phi_0^-$ is as in Proposition~\ref{prop:K}\,(ii), $u$ is as in~\eqref{eq:phi} (with $a_+=1$) and $\mu_\D$ is the smallest
	positive root of~\eqref{eq:secular}.
\end{proof}

\section{Proof of Proposition~\ref{prop:symm}} 
\label{app:prop}
\noindent\emph{Step 1.} For any $z\in\p\Omega$, the map
$\ov\Omega \ni y \to |y-z|$ is continuous. 
Hence, the maps defined as
\[
	 r_\rmi :=  \ov\Omg\ni  y \mapsto
 	\inf_{z\in\p\Omega}|y - z|,
	\qquad 
	r_\rmo := \ov\Omg\ni  y \mapsto
	 \sup_{z\in\p\Omega}|y-z|,
\]
are continuous on $\ov\Omega$ as well and they attain their upper and lower bounds. In particular, there exist $y_\rmi, y_\rmo \in \ov\Omega$ such that
\[
	r_\rmi(y_\rmi)=	\max_{y\in\ov\Omega}r_\rmi(y),
	\qquad 
	r_\rmo(y_\rmo)= \min_{y\in\ov\Omega}r_\rmo(y).
\]
\medskip

\noindent\emph{Step 2.} Assume that $\Omega$ has an axis of symmetry $\Lambda$. By Step 1 there exist $y_\rmi,y_\rmo \in \ov\Omega$ such that $r_\rmi(y_\rmi) = \sup_{y\in \Omega}r_\rmi(y)$ and $r_\rmo(y_\rmo) = \inf_{y\in \Omega}r_\rmo(y)$. 
Our aim is to show that $y_\rmi,y_\rmo$ can be both chosen in $\Lambda$. Let us suppose that $y_\rmi,y_\rmo\notin\Lambda$ and define the reflection $\cR_\Lambda\colon \ov\Omg\arr\ov\Omg$ with respect to $\Lambda$. Remark that $\cR_\Lambda y_\rmi$ and $\cR_\Lambda y_\rmo$ also satisfy $r_\rmi(\cR_\Lambda y_\rmi) = \max_{y\in \ov\Omega}r_\rmi(y)$ and $r_\rmo(\cR_\Lambda y_\rmo) = \min_{y\in \ov\Omega}r_\rmo(y)$. Set $\widetilde{y}_\rmi := \frac12 y_\rmi + \frac12 \cR_\Lambda y_\rmi$ and $\widetilde{y}_\rmo := \frac12 y_\rmo + \frac12 \cR_\Lambda y_\rmo$. As $\Omega$ is convex we have $\widetilde{y}_\rmi, \widetilde{y}_\rmo \in \ov\Omega$. Also by convexity of $\Omega$, we get 
\begin{equation}\label{eqn:convball}
	\frac12 \D_{r_\rmi(y_\rmi)}(y_\rmi) + \frac12\D_{r_\rmi(y_\rmi)}(\cR_\Lambda y_\rmi) = \D_{r_\rmi(y_\rmi)}(\wt{y}_\rmi) \subset \Omega,
\end{equation}
where $\D_r(y)$ denotes the disk of radius $r>0$
centred at $y\in \R^2$. Now, \eqref{eqn:convball} implies $r_\rmi(\widetilde{y}_\rmi) \geq r_\rmi(y_\rmi)$ and we obtain $r_\rmi(\widetilde{y}_\rmi) = \max_{y\in\ov\Omega} r_\rmi(y)$.

Similarly, by convexity of $\Omega$, we get
\begin{equation*}
\frac12 \D_{r_\rmo(y_\rmo)}(y_\rmo) + \frac12\D_{r_\rmo(y_\rmo)}(\cR_\Lambda y_\rmo) = \D_{r_\rmo(y_\rmo)}(\wt{y}_\rmo) \supset \Omega.
\end{equation*}
In particular, $r_\rmo(\widetilde{y}_\rmo) \leq \min_{y\in\ov\Omega}r_\rmo(y)$ and we have equality in this inequality.
\medskip

\noindent\emph{Step 3.}
Suppose now that $\Omega$ has two axes of symmetry $\Lambda_1$ and $\Lambda_2$. Let $y_\Lambda\in\Omg$ be the unique point of intersection of these axes. 
Thanks to Steps 1 and 2 for all $y\in \Omega$ we necessarily have $r_\rmi(y) \leq r_\rmi(y_\Lambda)$ and $r_\rmo(y) \geq r_\rmo(y_\Lambda)$.
Next, define the function
\[
	\cG(r_1,r_2)
	:= 
	\left(\frac{|\Omega|+ \pi r_1^2}{2\pi}\right)^{\frac12} \exp\left(2(r_\rmc -r_2)\Phi(r_1,r_\rmc)\right),\quad r_1 < r_2, r_\rmc < r_2.
\]
Remark that  $\cG$ is a non-decreasing function of $r_1$ whereas it is a non-increasing function of $r_2$.  Now, we have
\[
	\cF_\rmc(\Omega -y) 
	= 
	\cG\big(r_\rmi(y),r_\rmo(y)\big) 
	\leq 
	\cG\big(r_\rmi(y_\rmi),r_\rmo(y_\rmo)\big) 
	= 
	\cF_\rmc(\Omega - y_\Lambda).
\]
Hence, $\cF_\rmc^\star(\Omega) =\sup_{y\in\Omega} \cF_\rmc(\Omega-y)= \cF_\rmc(\Omega -y_\Lambda)$, by which the proof is concluded.

\end{document}